\documentclass[3p,times]{elsarticle}
\usepackage[colorlinks,linkcolor=blue,citecolor=blue]{hyperref}
\usepackage{bookmark}
\usepackage{bm}
\usepackage{amsmath,amsfonts,amsmath,amssymb,amsthm,cases}
\usepackage{algorithm,algorithmic}
\usepackage{multirow}
\usepackage{longtable}
\usepackage{fullpage}
\usepackage{cleveref}
\usepackage{graphics,graphicx,epsfig,epstopdf,subfigure}
\usepackage{color,natbib}
\usepackage[dvipsnames, svgnames, x11names]{xcolor}
\usepackage{mathtools}
\graphicspath{{Figures/}}

\newtheorem{theorem}{Theorem}[section]
\newtheorem{lemma}{Lemma}[section]

\newtheorem{remark}{Remark}[section]

\numberwithin{equation}{section}

\newcommand{\bb}{\boldsymbol}

\begin{document}
\begin{frontmatter}
	\title{A novel locking-free virtual element method for linear elasticity problems}
	\cortext[mycorrespondingauthor]{Corresponding author}
	\author[sjtu1]{Jianguo Huang\corref{mycorrespondingauthor}}
	\ead{jghuang@sjtu.edu.cn}
	\fntext[hjgfootnote]{The work of this author was partially supported by NSFC (Grant No.\ 12071289).}
	\author[sjtu1]{Sen Lin}
	\ead{sjtu\_Einsteinlin7@sjtu.edu.cn}
	\author[sjtu1,sjtu2]{Yue Yu}
	\ead{terenceyuyue@sjtu.edu.cn}
	\address[sjtu1]{School of Mathematical Sciences, and MOE-LSC,
		Shanghai Jiao Tong University, Shanghai 200240, China}
	\address[sjtu2]{Institute of Natural Sciences,
		Shanghai Jiao Tong University, Shanghai 200240, China}
	
	\begin{abstract}
	This paper devises a novel lowest-order conforming virtual element method (VEM) for planar linear elasticity with the pure displacement/traction boundary condition.
	The main trick is to view a generic polygon $K$ as a new one $\widetilde{K}$ with additional vertices consisting of interior points on edges of $K$,
	so that the discrete admissible space is taken as the $V_1$ type virtual element space related to the partition $\{\widetilde{K}\}$ instead of $\{K\}$.
	The method is shown to be uniformly convergent with the optimal rates both in $H^1$ and $L^2$ norms with respect to the Lam\'{e} constant $\lambda$.
	Numerical tests are presented to illustrate the good performance of the proposed VEM and confirm the theoretical results.
	\end{abstract}

	\begin{keyword}
		Virtual element method, Linear elasticity, Locking-free, Numerical tests.
	\end{keyword}		
	
\end{frontmatter}

\section{Introduction}
\label{Sec1}
Robust numerical solution of parameter dependent partial differential equations (PDEs) is a very important topic in scientific computing.
A typical example is the linear elasticity problem which involves the Lam\'{e} constants $\lambda$ and $\mu$.
If $\lambda$ is very large, the material becomes nearly incompressible, which will lead to the so-called volume locking phenomenon
when the Courant element is used for discretization (cf. \cite{Babuska-Szabo-1982}).
Because of this, a huge work has been developed to overcome this difficulty.
We refer the reader to \cite{Babuska-Suri-1992a,Babuska-Suri-1992b,BoffiBrezziFortin2013,Brenner-Sung-1992,Carstensen-Schedensack-2015,
Guo-Huang-2020,Harper-Wang-Liu-Tavener-Zhang-2020,Huang-Huang-2013,Wang-Wang-Wang-Zhang-2016,Zhang-1997}
and the references therein for details along this line.

On the other hand, the virtual element method (VEM) is a newly proposed numerical method for solving PDEs in recent years,
which can be viewed as a generalized finite element method allowing for the use of polytopal meshes (polygons or polyhedra) of very general shape.
It was first proposed and analyzed in \cite{Ahmad-Alsaedi-Brezzi-Marini-2013,Beirao-Brezzi-Cangiani-Manzini-2013, Beirao-Brezzi-Marini-Russo-2014}
and immediately attracted much attention from the research community due to its advantages in handling problems with complex geometries,
high-regularity solutions (cf. \cite{Chen-Huang-2020,Huang-2020}).
As far as we know, there have also developed some locking free VEMs for the linear elasticity problem.
Let $k$ denote the order of the virtual element used. The conforming VEMs for this problem were first studied in \cite{Beirao-Brezzi-Marini-2013}.
However, the condition $k\geq2$ is required so as to produce locking free optimal error estimates in the $H^1$ norm.
More recently, a lowest-order locking-free VEM was devised in \cite{Tang-Liu-Zhang-Feng-2020} to remedy this disadvantage
by modifying the $W_1$ type virtual element technically so that the uniform inf-sup condition holds. This method can be regarded as
the extension of the well-known Bernardi-Raugel finite element method (cf. \cite{Bernardi-Raugel-1985}) in general polygonal meshes.
In \cite{Zhang-Zhao-Yang-Chen-2019}, some nonconforming VEMs were proposed for attacking the previous problem in two and three dimensions.
It was shown in this paper that if $k\geq1$ (resp. $k\geq2$), the method is locking-free for the pure displacement (resp. traction) problem.
We mention in passing that there are some other VEMs discussing numerical solution of elasticity problems in different aspects; see, e.g.,
\cite{Adak-Nataraj-2020,Artioli-Miranda-Lovadina-Patruno-2017,Artioli-Miranda-Lovadina-Patruno-2018,
	Artioli-Miranda-Lovadina-Patruno-2020,Berbatov-Lazarov-Jivkov-2021,Caceres-Gatica-Sequeira-2019,
	Dassi-Lovadina-Visinoni-2020,Dhanush-Natarajan-2019,Gain-Talischi-Paulino-2014,Kwak-Park-2022,
	Mora-Rivera-2020,Reddy-Huyssteen-2019,Zhang-Feng-2018}.

In this paper, we are concerned with proposing and analyzing a novel lowest-order conforming locking-free VEM
for the pure dispalcement/traction problem, by means of a significant feature of VEMs. The main ideas behind the method include:
\begin{enumerate}
	\item
	For a generic polygonal element $K$ with $N_K$ edges, denoted by $\{z_i\}_{i=1}^{N_K}$ its vertices which are numbered in a counter-clockwise order
	and by $e_i=\overline{z_iz_{i+1}}$ the edge connecting $z_i$ to $z_{i+1}$, where $z_{N_K+1}:=z_1$.
	Then we construct an auxiliary polygon $\widetilde{K}$ (see Fig. \ref{Fig1_Polygon_Refine_Type3}),
	which has the same geometric shape as $K$ but has more vertices. Precisely speaking,
	besides the vertices of $K$,  $\widetilde{K}$ has additional vertices formed by the internal points $\{m_i\}_{i=1}^{N_K}$ of edges of $K$,
	which are chosen such that $|\overline{z_im_i}|=\alpha |\overline{z_iz_{i+1}}|$, where $\alpha\in (0,1)$ is any fixed constant.
	It is a significant advantage of VEMs that one can view an interior point of an edge as an additional vertex,
	so that the difficulty of hanging nodes during adaptivity can be overcome very conveniently (cf. \cite{Huang-Lin-2021}).
	But here we use the technique to construct locking free VEMs.
	
	\item
	Based on the $V_1$ type virtual element space with respect the partition $\{\widetilde{K}\}$ instead of the usual one $\{K\}$,
	a locking free VEM is naturally devised using the ideas in \cite{Beirao-Brezzi-Marini-2013} for $k\geq2$.
\end{enumerate}

Under the standard mesh assumption {\bf C0} (see Subsection \ref{Sec3-1}), we can show the proposed method is locking free
and has the optimal convergence order by invoking a novel interpolation operator which satisfies the required commutative property
and admits the optimal error estimates. We offer several numerical tests to illustrate the performance of the proposed method
and confirm its theoretical predictions.

We end this section by introducing some notations for later requirement.
We use standard notations for Sobolev spaces and norms (see \cite{Adams1975} for more details).
Let $D$ be any open subset of $\Omega$ and $|D|$ denote the area of $D$. Given a non-negative integer $s$,
we denote the standard Sobolev spaces on $D$ by $H^s(D)$ with norm $\|\cdot\|_{s,D}$ and semi-norm $|\cdot|_{s,D}$,
and denote the $L^2$-inner product on $D$ by $(\cdot,\,\cdot)_{D}$.
By convention, $H^0(D)=L^2(D)$.
Let $\mathbb{P}_s(D)$ be the set of polynomials of degree less than $s$ on $D$.
Moreover, we omit the subscript when $D=\Omega$.
For the vector-valued functions or spaces, we use the bold symbols, such as
$\bb{v}$, $\bb{L}^2(\Omega)$,  $\bb{H}^s(\Omega)$,
$\bb{H}_0^s(\Omega)$, etc.
Let $\boldsymbol{v}=(v_1,\,v_2)^{\intercal}$ and $\boldsymbol{\tau}=[\tau_{ij}]_{2\times2}$ be functions of two variables.
We define
\begin{align*}
{\rm{div}}~\bb{v} = \partial_1{v}_1 + \partial_2{v}_2, \quad {\rm{rot}}~\bb{v}= \partial_1{v}_2 - \partial_2{v}_1,
\quad \text{and} \quad
\boldsymbol{\rm{div}~\boldsymbol{\tau}}=(\partial_1{\tau_{11}}+\partial_2{\tau_{12}},\partial_1{\tau_{21}}+\partial_2{\tau_{22}})^{\intercal}.
\end{align*}
For any two quantities $a$ and $b$, ``$a\lesssim b$" indicates ``$a \leq Cb$"
with the hidden generic positive constant $C$ independent of $\lambda$ and the mesh size,
which may take different values at different occurrences, and $a\eqsim b$ shows $a\lesssim b\lesssim a$.

\section{The linear elasticity problem}
\label{Sec2}
Let $\Omega\subset\mathbb{R}^2$ be a convex polygonal domain with the boundary $\partial\Omega$.
Given an applied force $\bb{f}\in\boldsymbol{L}^2(\Omega)$, the linear elasticity problem is
to find the displacement vector $\bb{u}=(u_1,\,u_2)^{\intercal}$ such that
\begin{equation}
\label{Elasticity_Eq}
\begin{cases}
	- \boldsymbol{\rm div}~\bb{\sigma}(\bb{u}) = \bb{f} & \text{in} ~~\Omega, \\
	\bb{u}  = \bb{0}  & \text{on} ~~\partial\Omega,
\end{cases}
\end{equation}
where $\bb{\varepsilon}=[\varepsilon_{ij}]_{2\times2}$, $\varepsilon_{ij}(\boldsymbol{u})= \frac{1}{2}(\partial_i u_j + \partial_j u_i)$,
and
\[
	\bb{\sigma} = 2\mu\bb{\varepsilon}(\bb{u}) + \lambda({\rm div}~\bb{u})\bb{I}.
\]
Here, $\mu$ and $\lambda$ are the Lam\'{e} constants satisfying $\mu\in[\mu_1,\mu_2]$
with $0<\mu_1<\mu_2$ and $\lambda\in(0,\infty)$, and $\bb{I}$ is the identity matrix.

The variational formulation of \eqref{Elasticity_Eq} is
to find $\bb{u}\in\bb{V}=\bb{H}_0^1(\Omega)$ such that
\begin{equation}
\label{Elasticity_Eq_Var}
a(\bb{u},\,\bb{v}) = (\bb{f},\,\bb{v}),
\quad \bb{v} \in \bb{V},
\end{equation}
where
\begin{equation}
\label{Elasticity_Eq_Bilinear}
a(\bb{w},\,\bb{v})
= 2\mu (\bb{\varepsilon}(\bb{w}),\,\bb{\varepsilon}(\bb{v}))
+ \lambda ({\rm div}~\bb{w},\,{\rm div}~\bb{v}),
\quad \bb{w},\bb{v} \in \bb{V},
\end{equation}
and
\[
(\bb{f},\,\bb{v})
= \int_{\Omega} \bb{f}\bb{v}\,\mathrm{d}x,
\quad \bb{v} \in \bb{V}.
\]

By the well-known Korn inequality (cf. \cite{BrennerScott2008}):
\begin{align}
	\| \boldsymbol{v} \|_{1} \lesssim \| \bb{\varepsilon}(\boldsymbol{v}) \|_{0},
	\quad  \boldsymbol{v}\in \boldsymbol{H}_0^1(\Omega),
	\label{First_Korn_Iq}
\end{align}
it is easy to obtain the boundedness and coercivity of $a(\cdot,\,\cdot)$:
\begin{align*}
	a(\bb{w},\,\bb{v}) & \lesssim (1+\lambda)|\bb{w}|_1|\bb{v}|_1,
	\quad \bb{w},\bb{v}\in \bb{V},  \\
	| \bb{v} |_1^2 & \lesssim a(\bb{v},\,\bb{v}),
	\quad \bb{v}\in \bb{V},
\end{align*}
which imply that the problem \eqref{Elasticity_Eq} has a unique solution by the Lax-Milgram lemma.

Since the first term of $a(\cdot,\cdot)$ (cf. \eqref{Elasticity_Eq_Bilinear}) is related to $\mu$, we write
\[
	a_\mu( \bb{w}, \bb{v}) = (\bb{\varepsilon}(\bb{w}),\,\bb{\varepsilon}(\bb{v}))
\]
for later uses.

It is shown in \cite{Brenner-Sung-1992} that the following regularity estimate holds for the solution $\bb{u}$ of \eqref{Elasticity_Eq}:
\begin{equation}
	\label{Elasticity_Eq_Reglar_Est}
	\|\bb{u}\|_2 + \lambda\|{\rm div}~\bb{u}\|_1 \lesssim \|\bb{f}\|_0.
\end{equation}

\section{Virtual element methods for the pure displacement problem }
\label{Sec3}
In this section, by means of a mesh refinement trick in VEMs,
we propose a novel lowest order locking-free conforming virtual element method on the fine mesh.

\subsection{Mesh assumption and the virtual element space}
\label{Sec3-1}
Let $\{\mathcal{T}_h\}$ be a family of decompositions of $\Omega$ into polygonal elements.
The generic element is denoted by $K$ with diameter $h_K={\rm diam}(K)$.
For clarity, we work on meshes satisfying the following condition (cf. \cite{Beirao-Brezzi-Cangiani-Manzini-2013}):
\begin{enumerate}
	\item [{\bf C0}.]
	For each element $K$, there exists positive constants $\gamma_1$ and $\gamma_2$ independent of $h_K$ such that
	\begin{itemize}
		\item
		$K$ is star-shaped with respect to a disc in $K$ with radius $\geq \gamma_1 {h_K}$;
		
		\item
		the distance between any two vertices of $K$ is $\geq\gamma_2 h_K$.
	\end{itemize}
\end{enumerate}

\begin{figure}[H]
	\centering
	\includegraphics[scale=0.27]{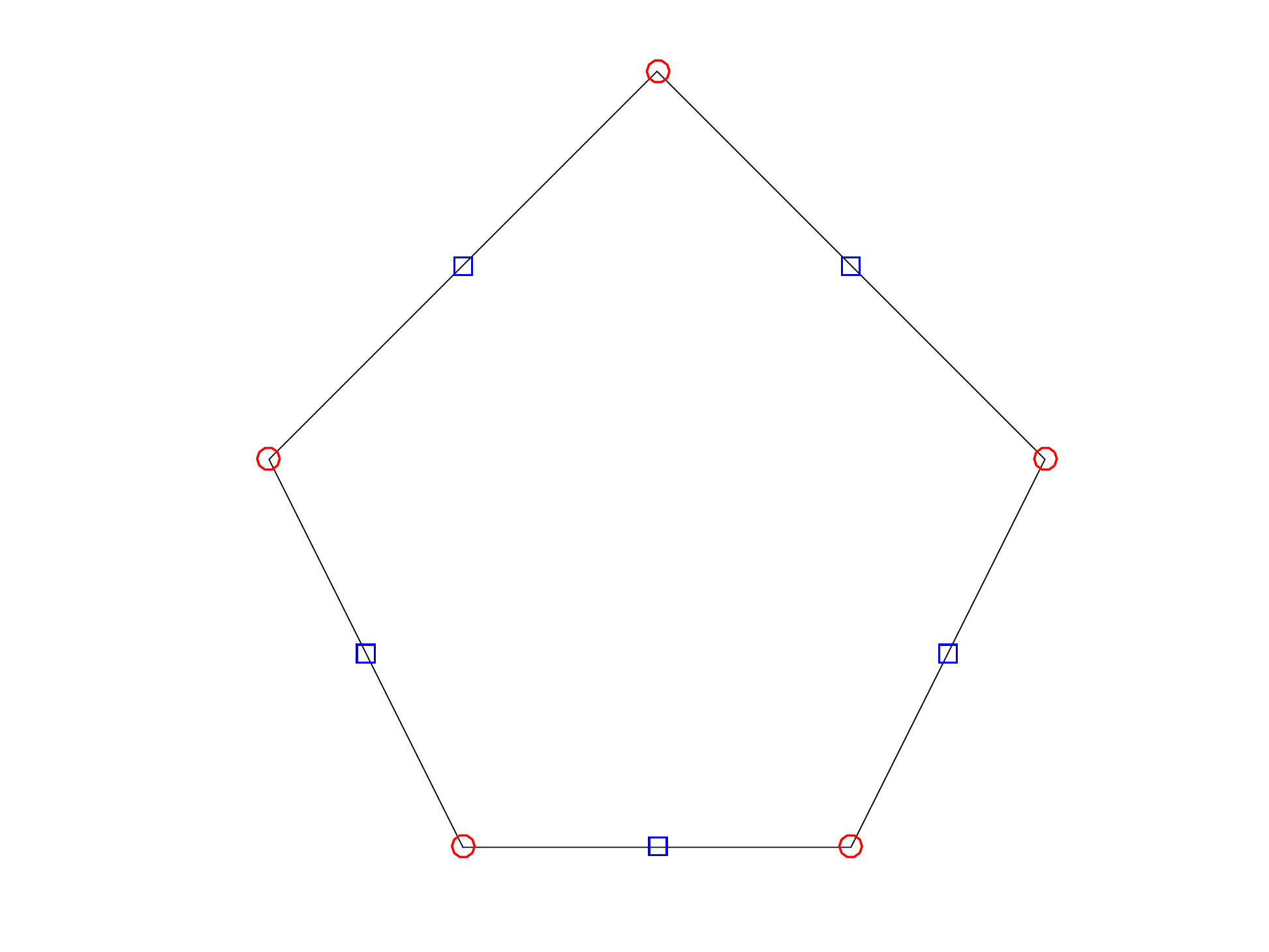}
	\caption{
		The auxiliary polygon $\widetilde{K}$ satisfying the mesh assumption {\bf C0}.
		The polygonal elements of the fine mesh $\mathcal{T}_h^{*}$ and the original mesh $\mathcal{T}_h$ are denoted by $\widetilde{K}$ and $K$, respectively.
		The red circle points denote the original vertices, the blue square points denote the added vertices.
	}
	\label{Fig1_Polygon_Refine_Type3}
\end{figure}

Due to the large flexibility of the meshes in VEMs, under the above assumption,
one easily finds that the mesh can be refined by taking the internal point $m_i$ of each edge $e_i$ of every element $K$
as a new vertex (see Fig.~\ref{Fig1_Polygon_Refine_Type3}).
For the new edge $|\overline{z_im_i}|$, we assume $|\overline{z_im_i}|=\alpha|e_i|$ for $i=1,2,\ldots,N_K$ for a given constant $\alpha\in(0,1)$.
The resulting fine mesh will be denoted by $\mathcal{T}_h^{*}$ with the generic element given by $\widetilde{K}$.
It's obvious that $\mathcal{T}_h^{*}$ still satisfies the assumption {\bf C0} thought the generic constants may change.
Here, $h=\max_{\widetilde{K}\in\mathcal{T}_h^{*}}h_{\widetilde{K}}$ with $h_{\widetilde{K}}$ being the diameter of $\widetilde{K}$.
It is evident that $h_K\eqsim h_{\widetilde{K}}$.

In this paper, we consider the lowest-order conforming virtual element space (VES) defined on the fine mesh $\mathcal{T}_h^*$.
To this end, we briefly review the construction method mentioned in \cite{Beirao-Brezzi-Cangiani-Manzini-2013}.
The local scalar virtual element space on $\widetilde{K}$ is defined as
\begin{equation}
\label{Elasticity_Eq_Local_VES}
V_1(\widetilde{K}) =\big\{
v \in H^1(\widetilde{K}) ~|~ \Delta v = 0 ~~\mbox{in}~~\widetilde{K},
~ v |_e \in \mathbb{P}_1(e),\quad e \subset \partial{\widetilde{K}}\big\}.
\end{equation}
The local virtual element space for linear elasticity is then given by
the tensor product space $\bb{V}_1(\widetilde{K}) = (V_1(\widetilde{K}))^2$.
The corresponding local degrees of freedoms (DOFs) are:
\begin{itemize}
	\item
	$\chi_i\in \bb{\chi}$:  the values at vertices of $\widetilde{K}$,
	\[
	\chi_i(\bb{v}) = \bb{v}(a_i),
	\quad \mbox{where~$a_i=z_i$~or~$m_i$~is the vertex of ~$\widetilde{K}$}.
	\]
\end{itemize}

In what follows, denote by $\bb{V}_h$ the global VES, which is defined elementwise and required that
the degrees of freedom are continuous through the interior vertices while zero for the boundary DOFs,
i.e.,
\begin{align*}
\bb{V}_h = \big\{
\bb{v}\in \bb{H}_0^1(\Omega) ~|~ \bb{v}|_{\widetilde{K}} \in \bb{V}_1(\widetilde{K}), \quad \widetilde{K}\in\mathcal{T}_h^{*}
\big\}.
\end{align*}

\subsection{The virtual element method}
\label{Sec3-2}
As in \cite{Beirao-Brezzi-Marini-2013},
we first introduce the elliptic projection ${\Pi}_1^{\widetilde{K}} \colon \bb{H}^1(\widetilde{K})\to (\mathbb{P}_1(\widetilde{K}))^2$
for dealing with the first term of \eqref{Elasticity_Eq_Bilinear}, which is given by
	\begin{equation}
	\label{Pi1E_Proj}
	\begin{cases}
		(\bb{\varepsilon}({\Pi}_1^{\widetilde{K}}\bb{v}),\,\bb{\varepsilon}(\bb{p}))_{\widetilde{K}}
		= (\bb{\varepsilon}(\bb{v}),\,\bb{\varepsilon}(\bb{p}))_{\widetilde{K}},
		\quad \bb{p}\in (\mathbb{P}_1(\widetilde{K}))^2, \\
	\int_{\partial{\widetilde{K}}}{\Pi}_1^{\widetilde{K}}\bb{v}\,\mathrm{d}s
		= \int_{\partial{\widetilde{K}}}\bb{v}\,\mathrm{d}s, \\
	\int_{\widetilde{K}}{\rm{rot}}~{\Pi}_1^{\widetilde{K}} \bb{v}\,\mathrm{d}x
		= \int_{\widetilde{K}}{\rm{rot}}~\bb{v}\,\mathrm{d}x.
	\end{cases}
	\end{equation}
One can check that the elliptic projection ${\Pi}_1^{\widetilde{K}}$ can be computed by the given DOFs.
For simplicity, we write $\Pi_1$ for the related elementwise defined global operator.

Next, Let $\Pi_0^K$ be the usual $L^2$ orthogonal projection from $L^2(K)$ into the constant space over $K$. In other words,
for all $v\in L^2(K)$, $\Pi_0^K v=\int_K v \mathrm{d}x/|K|$. This definition naturally applies to a $L^2$ vector valued function.
For all $\bb{v}\in\bb{H}^1(K)$, by integration by parts,
\[
\int_K {\rm div}~\bb{v}\,\mathrm{d}x =  \int_{\partial{K}} \bb{v}\cdot \bb{n}_K\, \mathrm{d}s
\quad \mbox{and} \quad
\int_K {\rm rot}~\bb{v}\,\mathrm{d}x =  \int_{\partial{K}} \bb{v}\cdot \bb{t}_K\, \mathrm{d}s,
\]
where $\bb{n}_K=(n_1,n_2)^{\intercal}$ and $\bb{t}_K=(-n_2,n_1)^{\intercal}$ denote the unit normal and tangent vector
to $\partial{K}$ determined by the right hand side, respectively. Hence, for all $\bb{v} \in \bb{V}_1(K)$,
we can compute $\Pi_0^K{\rm div}~\bb{v}$  and $\Pi_0^K{\rm rot}~\bb{v}$ only using the function values at the vertices of $K$.
Therefore, we treat the second term in \eqref{Elasticity_Eq_Bilinear} using the exact method for the case $k\geq 2$ given in \cite{Beirao-Brezzi-Marini-2013}.

With the help of the above projections, we construct the local computable bilinear form by
\begin{align}
a_h^{\widetilde{K}} (\bb{w}_h,\,\bb{v}_h)
= 2 \mu  a_{\mu,h}^{\widetilde{K}} (\bb{w}_h,\,\bb{v}_h)
+\lambda(\Pi_0^K{\rm div}~\bb{w}_h,\,\Pi_0^K{\rm div}~\bb{v}_h)_K,
\quad \bb{w}_h, \bb{v}_h \in \bb{V}_1(\widetilde{K}),
\label{Elasticity_Eq_Local_Discrete_Bilinear}
\end{align}
where
\begin{align*}
a_{\mu,h}^{\widetilde{K}}(\bb{w}_h,\,\bb{v}_h)
= (\bb{\varepsilon}(\Pi_1^{\widetilde{K}}\bb{w}_h),\,
\bb{\varepsilon}(\Pi_1^{\widetilde{K}}\bb{v}_h))_{\widetilde{K}}
+S^{\widetilde{K}}(\bb{w}_h-\Pi_1^{\widetilde{K}}\bb{w}_h,\,\bb{v}_h-\Pi_1^{\widetilde{K}}\bb{v}_h)
\end{align*}
and
\[
	S^{\widetilde{K}}(\bb{w}_h,\,\bb{v}_h) = \bb{\chi}(\bb{w}_h) \cdot \bb{\chi}(\bb{v}_h)
	= \sum\limits_{i=1}^{N_{\widetilde{K}}} \chi_i(\bb{w}_h) \cdot \chi_i(\bb{v}_h)
\]
with $N_{\widetilde{K}}=2N_K$. 

The locking-free conforming VEM of the variational problem \eqref{Elasticity_Eq_Var} is
to find $\bb{u}_h\in\bb{V}_h$ such that
\begin{equation}
\label{Elasticity_Eq_VEM}
a_h(\bb{u}_h,\,\bb{v}_h) = \langle \bb{f}_h,\,\bb{v}_h \rangle, \quad  \bb{v}_h\in\bb{V}_h,
\end{equation}
where
\[
a_h(\bb{u}_h,\,\bb{v}_h) = \sum\limits_{K\in \mathcal{T}_h}a_h^K(\bb{u}_h,\,\bb{v}_h),
\]
and the approximation of the right hand side is given by
\begin{equation}
\label{Elasticity_Eq_VEM_rhs}
\langle \bb{f}_h,\,\bb{v}_h \rangle = \sum_{\widetilde{K}\in\mathcal{T}_h^{*}}(\bb{f}_h,\, \hat{\bb{v}}_h)_{\widetilde{K}},
\end{equation}
where $\bb{f}_h|_{\widetilde{K}} = \Pi_0^{\widetilde{K}}\bb{f}$ and
\[
	\hat{\bb{v}}_h = \frac{1}{|\partial{\widetilde{K}}|} \int_{\partial\widetilde{K}} \bb{v}_h \,\mathrm{d}s.
\]

\begin{remark}
	\label{Remk1}
	It is worth mentioning that the proposed lowest-order conforming virtual element on triangular meshes is not the well-known Courant element.
\end{remark}

\subsection{Stability of the virtual element method}
\label{Sec3-3}
We first recall a classic Korn's inequality with respect to a star-shaped domain.
\begin{lemma} [\cite{Brenner-Sung-1992,Costabel-Dauge-2015}]
	\label{Lem1_Rot_Korn_Iq}
	Let $D\subset \mathbb{R}^2$ be a bounded domain of diameter $h$, which is star-shaped with respect to a disc of radius $\rho$.
	Then for any $\boldsymbol v\in \boldsymbol H^1(D)$ satisfying $\int_D {\rm {rot}}~\bb{v}\,\mathrm{d}x=0$,
	there holds
	\begin{align}
	| \boldsymbol{v} |_{1,D} \lesssim \| \boldsymbol{\varepsilon}(\boldsymbol{v}) \|_{0,D},	
	\nonumber
	\end{align}
	where the hidden constant $C$ only depends on the aspect ratio $h/\rho$.
\end{lemma}

With the help of the above lemmas, we are able to derive the norm equivalence and the stability condition
for the proposed virtual element method \eqref{Elasticity_Eq_VEM} described as follows.
\begin{theorem}
\label{Thm1_Elasticity_Norm_Stab}
Under the mesh assumption \textbf{C0}, for all $\bb{v}_h \in \bb{V}_h$ there hold
\begin{align}
\| \bb{\varepsilon}(\bb{v}_h - \Pi_1^{\widetilde{K}} \bb{v}_h) \|_{0,\widetilde{K}} \eqsim \|\bb{\chi}(\bb{v}_h - \Pi_1^{\widetilde{K}} \bb{v}_h)\|_{l^2},
\quad \widetilde{K} \in \mathcal{T}_h^{*},
\label{Elasticity_Norm_Equiv}
\end{align}
and
\begin{align}
a_{\mu,h}^{\widetilde{K}}(\bb{v}_h,\,\bb{v}_h) \eqsim a_{\mu}^{\widetilde{K}}(\bb{v}_h,\,\bb{v}_h),
\quad \widetilde{K} \in \mathcal{T}_h^{*} ,
\label{Elasticity_Stability_Cond}
\end{align}
where $a_{\mu}^{\widetilde{K}}(\bb{v}_h,\,\bb{v}_h)$ is the local conterpart of $a_{\mu}(\bb{v}_h,\,\bb{v}_h)$.
\end{theorem}

\begin{proof}
By the definition of $\Pi_1^{\widetilde{K}}$ and Lemma \ref{Lem1_Rot_Korn_Iq}, one has
\begin{equation}
\label{Elasticity_Norm_Equiv_aux1}
\| \bb{\varepsilon}(\bb{v}_h - \Pi_1^{\widetilde{K}} \bb{v}_h) \|_{0,\widetilde{K}} \eqsim  |\bb{v}_h - \Pi_1^{\widetilde{K}} \bb{v}_h|_{1,\widetilde{K}},
\quad \bb{v}_h \in \bb{V}_h.
\end{equation}
According to the norm equivalence in \cite{Chen-Huang-2018},
\[
	h_{\widetilde{K}}^{ - 1}\| \boldsymbol v \|_{0,\widetilde{K}}  \eqsim \|\boldsymbol\chi(\boldsymbol v)\|_{l^2}
\]
for any $\bb{v} \in \bb{V}_1(\widetilde{K})$, and hence
\begin{equation}
\label{Elasticity_Norm_Equiv_aux2}
h_{\widetilde{K}}^{- 1}\| \bb{v}_h - \Pi_1^{\widetilde{K}} \bb{v}_h \|_{0,\widetilde{K}}   \eqsim \|\boldsymbol\chi(\bb{v}_h - \Pi_1^{\widetilde{K}} \bb{v}_h)\|_{l^2}.
\end{equation}
In view of the inverse inequality in \cite{Chen-Huang-2018} and the standard Poincar\'{e}-Friedrichs inequality,
we further obtain
\[
	h_{\widetilde{K}}^{ - 1}\| \bb{v}_h - \Pi_1^{\widetilde{K}} \bb{v}_h \|_{0,\widetilde{K}}
	\eqsim |\bb{v}_h - \Pi_1^{\widetilde{K}} \bb{v}_h|_{1,\widetilde{K}},
\]
which together with \eqref{Elasticity_Norm_Equiv_aux1} and \eqref{Elasticity_Norm_Equiv_aux2} yields the first estimate.

The second estimate follows readily from the first estimate by noting that
\begin{align*}
a_{\mu}^{\widetilde{K}}(\bb{v}_h,\,\bb{v}_h)
	&= \|\bb{\varepsilon}(\bb{v}_h)\|_{0,\widetilde{K}}^2
	= \|\bb{\varepsilon}(\Pi_1^{\widetilde{K}} \bb{v}_h)\|_{0,\widetilde{K}}^2 + \|\bb{\varepsilon}(\bb{v}_h - \Pi_1^{\widetilde{K}} \bb{v}_h)\|_{0,\widetilde{K}}^2.
\end{align*}
\end{proof}

\section{Error analysis}
\label{Sec4}
Following the similar arguments in \cite{Beirao-Brezzi-Marini-2013}, we can derive an abstract lemma for the error analysis.
\begin{lemma}	
\label{Lem2_Elasticity_Eq_VEM_Conver_Est}
Let  $\bb{u}\in\bb{V}$ and $\bb{u}_h\in\bb{V}_h$ be the solution of the problem \eqref{Elasticity_Eq_Var} and
\eqref{Elasticity_Eq_VEM}, respectively.
Then under the mesh assumption \textbf{C0}, for any $\bb{u}_{\mathcal{I}} \in \bb{V}_h$
and for any piecewise polynomial $\bb{u}_{\pi}\in (\mathbb{P}_k(\mathcal{T}_h^*))^2$, there holds
	\begin{equation}
	\label{Elasticity_VEM_Conver_Est}
	| \bb{u} - \bb{u}_h |_1
	\lesssim \big( | \bb{u} - \bb{u}_{\mathcal{I}} |_1
	+ | \bb{u} - \bb{u}_{\pi} |_{1,h}
	+ \lambda \| {\rm div}~\bb{u} - \Pi_0{\rm div}~\bb{u}_{\mathcal{I}} \|_0
	+ \| \bb{f} - \bb{f}_h \|_{\bb{V}_h^{'}}\big),
	\end{equation}
	where $\Pi_0{\rm div}~\bb{v} |_K = \Pi_0^K{\rm div}~\bb{v}$ for all $\bb{v}\in \bb{V}$,	and
	\begin{align}
	\| \bb{f} - \bb{f}_h \|_{\bb{V}_h^{'}}
	= \sup_{\bb{v}\in\bb{V}_h}
	\frac{(\bb{f},\,\bb{v}_h) - \langle\bb{f}_h,\,\bb{v}_h\rangle}{|\bb{v}_h|_1}.
	\label{Elasticity_VEM_rhs_Dual_Norm}
	\end{align}
\end{lemma}
\begin{proof}
By setting $\bb{\delta}_h = \bb{u}_h-\bb{u}_{\mathcal{I}}$ and using the triangle's inequality, it suffices to bound $|\bb{\delta}_h|_1$.
According to the equivalence \eqref{Elasticity_Stability_Cond} and Korn's inequality \eqref{First_Korn_Iq}, we have the coercivity
\[
	|\bb{\delta}_h|_1^2
	\lesssim  \|\bb{\varepsilon}(\bb{\delta}_h)\|_0^2
	=  a_{\mu}(\bb{\delta}_h,\,\bb{\delta}_h)
	\eqsim a_{\mu,h}(\bb{\delta}_h,\,\bb{\delta}_h).
\]
The remaining arguments are similar to the ones given in the proof of Theorem 3.2 in \cite{Beirao-Brezzi-Marini-2013}.
\end{proof}

We now prove the locking-free property of the proposed VEM.
To this end, we shall introduce a special interpolation operator as presented in the proof of the following lemma.
\begin{lemma}
\label{Lem3_Elasticity_Interp_Prop}
Let $\bb{u} \in \bb{H}_0^1(\Omega)\cap\bb{H}^2(\Omega)$ be the solution of problem \eqref{Elasticity_Eq_Var}.
Then there exists a VEM function $\bb{u}_{\mathcal{I}} \in \bb{V}_h$ such that
\begin{align}
\Pi_0^K{\rm div}~\bb{u}_{\mathcal{I}}  = \Pi_0^K{\rm div}~\bb{u}, \quad
\Pi_0^K{\rm rot}~\bb{u}_{\mathcal{I}}  = \Pi_0^K{\rm rot}~\bb{u},
\quad K\in\mathcal{T}_h,
\label{Elasticity_Interp_uI_Prop}
\end{align}
and
\begin{align}
| \bb{u} - \bb{u}_{\mathcal{I}} |_{1,K} \lesssim h |\bb{u}|_{2,K},
\quad K\in\mathcal{T}_h.
\label{Elasticity_Interp_uI_Est}
\end{align}
\end{lemma}
\begin{proof}
By the definition of $\Pi_0^K{\rm div}$ and by integration by parts, one has
\begin{align*}
\Pi_0^K{\rm div}~\bb{u} = \frac{1}{|K|} \int_K {\rm div}~\bb{u}\,\mathrm{d}x
= \frac{1}{|K|}\int_{\partial{K}}\bb{u}\cdot\bb{n}_K\,\mathrm{d}s
=\frac{1}{|K|} \sum_{e\subset\partial{K}}\int_{e}\bb{u}\cdot\bb{n}_K\,\mathrm{d}s.
\end{align*}
Then the first relation of \eqref{Elasticity_Interp_uI_Prop} is equal to
\begin{align*}
\sum_{e\subset\partial{K}}\int_{e}\bb{u}_{\mathcal{I}}\cdot\bb{n}_K\,\mathrm{d}s
=\sum_{e\subset\partial{K}}\int_{e}\bb{u}\cdot\bb{n}_K\,\mathrm{d}s.
\end{align*}
Similarly, the second relation of \eqref{Elasticity_Interp_uI_Prop} is equal to
\begin{align*}
\sum_{e\subset\partial{K}}\int_{e}\bb{u}_{\mathcal{I}}\cdot\bb{t}_K\,\mathrm{d}s
=\sum_{e\subset\partial{K}}\int_{e}\bb{u}\cdot\bb{t}_K\,\mathrm{d}s.
\end{align*}
Thus, it suffices to construct a VEM function $\bb{u}_{\mathcal{I}} \in \bb{V}_h$ such that
\[
\int_{e}\bb{u}_{\mathcal{I}}\,\mathrm{d}s
=\int_{e}\bb{u}\,\mathrm{d}s,
\quad e\subset\partial{K}.
\]

We construct the function $\bb{u}_{\mathcal{I}}$ by determining its DOF values.
For any edge $e_i$ of $K$ with two end points $z_i$ and $z_{i+1}$ with $i=1,2,\ldots,N_K$,
there exists one internal point $m_i$ such that the edge $e_i$ is divided into two subedges
$e_{i1}=\overline{z_im_i}$ and $e_{i2}=\overline{m_iz_{i+1}}$, and
$|e_{i1}|=\alpha|e_i|$. For the vertices of $K$, we define
\[
	\bb{u}_{\mathcal{I}}(z_i)=\bb{u}(z_i), \quad i=1,2,\ldots,N_K.
\]
Since  $\bb{u}_{\mathcal{I}}$ is piecewise linear along $\partial K$, we have
\begin{align*}
\int_{e_i}\bb{u}\,\mathrm{d}s
&= \int_{e_i}\bb{u}_{\mathcal{I}}\,\mathrm{d}s
	= \int_{e_{i1}}\bb{u}_{\mathcal{I}}\,\mathrm{d}s+\int_{e_{i2}}\bb{u}_{\mathcal{I}}\,\mathrm{d}s \\
&= \frac{|e_{i1}|}{2}\big(\bb{u}_{\mathcal{I}}(z_i)+\bb{u}_{\mathcal{I}}(m_i)\big) +
	\frac{|e_{i2}|}{2}\big(\bb{u}_{\mathcal{I}}(m_i)+\bb{u}_{\mathcal{I}}(z_{i+1})\big) \\
&= \frac{|e_i|}{2}\big(\bb{u}_{\mathcal{I}}(m_i)+\alpha\bb{u}_{\mathcal{I}}(z_i)+(1-\alpha)\bb{u}_{\mathcal{I}}(z_{i+1})\big) \\
&= \frac{|e_i|}{2}\big(\bb{u}_{\mathcal{I}}(m_i)+\alpha\bb{u}(z_i)+(1-\alpha)\bb{u}(z_{i+1})\big),
\end{align*}
which forces us to define
\[
	\bb{u}_{\mathcal{I}}(m_i)
	=\frac{2}{|e_i|}\int_{e_i}\bb{u}\,\mathrm{d}s - \alpha\bb{u}(z_i) -  (1-\alpha)\bb{u}(z_{i+1}),
	\quad i=1,2,\ldots,N_K.
\]

Denote $\bb{u}_{\mathcal{I}}|_K= \mathcal{I}_K\bb{u}$ for all $K\in\mathcal{T}_h$.
Since $\mathcal{I}_K\bb{u}\in \bb{V}_1(\widetilde{K})$,  we have by the inverse inequality and norm equivalence in \cite{Chen-Huang-2018} that
\begin{equation}
\label{Elasticity_Interp_uI_Est_aux1}
|\mathcal{I}_K\bb{u} |_{1,K}^2 \lesssim h_K^{-2} \| \mathcal{I}_K\bb{u} \|_{0,K}^2
\eqsim \sum_{p\in K}(\mathcal{I}_K\bb{u}(p))^2\lesssim \|u\|_{0,\infty,K}^2,
\end{equation}
where $p\in K$ means $p$ is a vertex of $\widetilde{K}$.

Under the mesh assumption {\bf C0}, by connecting the center of the disc used in the star-shaped condition with the vertices of $K$,
we can obtain a virtual triangulation $\mathcal{T}_K$ of $K$, which is quasi-uniform and shape-regular with the mesh size $h_K$.
Moreover, the number of the triangles in $\mathcal{T}_K$ is uniformly bounded with respect to $\mathcal{T}_h$.
For all $\tau\in \mathcal{T}_K$, it follows from the Sobolev inequality and the scaling argument that
\[
	\|\bb{u} \|_{0,\infty,\tau}^2\lesssim h_{\tau}^{-2}\| \bb{u} \|_{0,\tau}^2 + | \bb{u} |_{1,\tau}^2 + h_{\tau}^2 | \bb{u} |_{2,\tau}^2,
\]
where $h_{\tau}$ denotes the diameter of $\tau$, which is proportional to $h_K$. This immediately implies
\begin{equation}
\label{Elasticity_Interp_uI_Est_aux2}
\|\bb{u} \|_{0,\infty,K}^2\lesssim h_K^{-2}\| \bb{u} \|_{0,K}^2 + | \bb{u} |_{1,K}^2 + h_K^2 | \bb{u} |_{2,K}^2,\quad \bb{u}\in\bb{H}^2(K).
\end{equation}

On the other hand, by the Dupont-Scott theory (cf. \cite{BrennerScott2008}), for all $\bb{u}\in\bb{H}^2(K)$
there exists an $\bb{u}_{\pi}\in (\mathbb{P}_1(K))^2$ such that
\begin{align}
\| \bb{u}-\bb{u}_{\pi} \|_{0,K} + h_K | \bb{u}-\bb{u}_{\pi} |_{1,K}
\lesssim h_K^2 | \bb{u} |_{2,K}.
\label{Elasticity_Interp_uI_Est_aux3}
\end{align}

Observing that $ \mathcal{I}_K\bb{u}_{\pi}=\bb{u}_{\pi}$, we have by \eqref{Elasticity_Interp_uI_Est_aux1}--\eqref{Elasticity_Interp_uI_Est_aux3} that
\begin{align*}
| \bb{u}-\bb{u}_{\mathcal{I}} |_{1,K}
&= | \bb{u}- \mathcal{I}_K\bb{u} |_{1,K}
\leq | \bb{u}-\bb{u}_{\pi}|_{1,K} + | \mathcal{I}_K(\bb{u}-\bb{u}_{\pi}) |_{1,K} \\
& \lesssim h_K | \bb{u} |_{2,K} + h_K^{-1} \| \bb{u}-\bb{u}_{\pi} \|_{0,K} +  | \bb{u}-\bb{u}_{\pi}|_{1,K} + h_K | \bb{u}-\bb{u}_{\pi} |_{2,K} \\
& \lesssim h_K  | \bb{u} |_{2,K}.
\end{align*}
So the required estimate \eqref{Elasticity_Interp_uI_Est} holds.
\end{proof}

Thanks to Lemmas \ref{Lem2_Elasticity_Eq_VEM_Conver_Est}-\ref{Lem3_Elasticity_Interp_Prop},
the projection error estimates (cf. \cite{BrennerScott2008}) and
the estimate for \eqref{Elasticity_VEM_rhs_Dual_Norm} (cf. \cite{Beirao-Brezzi-Marini-2013}),
we are able to derive the following error estimate which is uniform with the incompressible parameter $\lambda$.
\begin{theorem}
	\label{Thm2_Elasticity_Eq_VEM_H1_Est}
	Let $\bb{u}$ and $\bb{u}_h$ be given as in Lemma \ref{Lem2_Elasticity_Eq_VEM_Conver_Est}.
	Under the mesh assumption {\bf C0} and the condition of regularity estimate \eqref{Elasticity_Eq_Reglar_Est},
	there holds
	\begin{align}
	| \bb{u} - \bb{u}_h |_1 \lesssim h \| \bb{f} \|_0,
	\label{Elasticity_Eq_VEM_H1_Err_Est}
	\end{align}
	with the hidden positive constant $C$ independent of $h$ and $\lambda$.
\end{theorem}

Using the duality argument and following the similar proof of Lemma 2.4 in \cite{Beirao-Brezzi-Marini-2013},
we can establish the error estimate in $L^2$ norm.
\begin{theorem}
	\label{Thm3_Elasticity_Eq_VEM_L2_Est}
	Assume that the domain $\Omega$ is convex. Under the same assumptions of Theorem \ref{Thm2_Elasticity_Eq_VEM_H1_Est},
	there holds
	\begin{align}
	\| \bb{u} - \bb{u}_h \|_0 \lesssim h^2 \| \bb{f} \|_0.
	\label{Elasticity_Eq_VEM_L2_Err_Est}
	\end{align}
\end{theorem}

\section{The VEM for the pure traction problem}
\label{Sec5}
In this section, we extend our approach to the pure traction problem described as follows:
\begin{equation}
\label{Elasticity_Eq_Traction}
\begin{cases}
- {\mathbf{div}}~\bb{\sigma}(\bb{u}) = \bb{f} & \text{in} ~~\Omega, \\
\bb{\sigma}(\bb{u}) \bb{n}  = \bb{0}  & \text{on} ~~\partial\Omega,
\end{cases}
\end{equation}
where $\bb{n}$ denotes the unit vector outward normal to $\partial\Omega$.
The problem \eqref{Elasticity_Eq_Traction} is uniquely solvable in $\widehat{\bb{V}}$ defined later
if  and only if $\bb{f}$ satisfies the following compatibility condition (cf. \cite{BrennerScott2008,Brenner-Sung-1992})
\begin{align}
\int_{\Omega} \bb{f}\cdot \bb{v}\,\mathrm{d}x = 0,
\quad \bb{v}\in{\bf{\rm RM}},
\nonumber
\end{align}
where
\[{\bf{\rm RM}} = \big\{\bb{v}~|~ \bb{v}=(a+bx_2,\,c-bx_1)^{\intercal}, ~a,b,c\in\mathbb{R}\big\}.\]

We define the admissible space as
\begin{align*}
\widehat{\bb{V}} = \Big\{
\bb{v} \in \bb{H}^1(\Omega) ~\big|~
\int_{\partial\Omega} \bb{v}\,\mathrm{d}s = \bb{0},
~ \int_{\Omega} {\rm rot} ~\bb{v} \,\mathrm{d}x=0
\Big\},
\end{align*}
which is slightly different from the space $\widehat{\boldsymbol{H}}^1(\Omega)$ defined by (cf. \cite{BrennerScott2008})
\[
	\widehat{\boldsymbol{H}}^1(\Omega) = \Big\{
	\bb{v} \in \bb{H}^1(\Omega) ~\big|~
	\int_{\Omega} \bb{v}\,\mathrm{d}x = \bb{0},
	~ \int_{\Omega} {\rm rot} ~\bb{v} \,\mathrm{d}x=0
	\Big\}.
\]
Here, the first constraint is replaced by an integrand on the boundary $\partial\Omega$ not the whole domain $\Omega$.
The choice here is considered for two reasons, one is to ensure the interpolation of the VEM function lies in the underlying VES,
and the other is to make the terms associated with Lagrange multipliers computable in the implementation.
One can refer to Subsection \ref{Sec7-1} for details.
The continuous variational problem \eqref{Elasticity_Eq_Traction} is the same as the one
given in \eqref{Elasticity_Eq_Var}, with  the function space replaced by $\widehat{\bb{V}}$.
The local VES on $\widetilde{K}$  is given by $\widehat{\boldsymbol{V}}_1(\widetilde{K})=(\widehat{V}_1(\widetilde{K}))^2$,
where $\widehat{V}_1(\widetilde{K})$ is the same as the one given in \eqref{Elasticity_Eq_Local_VES}, i.e.
\begin{equation*}
\widehat{V}_1(\widetilde{K}) =\big\{
v \in H^1(\widetilde{K}) ~|~ \Delta {v} = 0 ~~\mbox{in}~~\widetilde{K},
~ {v} |_e \in \mathbb{P}_1(e), \quad e \subset \partial{\widetilde{K}}\big\},
\end{equation*}
and the global VES given by
\begin{align*}
\widehat{\bb{V}}_h = \Big\{
\bb{v}_h\in\bb{H}^1(\Omega) ~|~ \bb{v}_h|_{\widetilde{K}} \in \widehat{\bb{V}}_1(\widetilde{K}),
\quad \widetilde{K}\in\mathcal{T}_h, ~
\int\nolimits_{\partial\Omega} \bb{v}_h\,\mathrm{d}s=\bb{0}, ~ \int_{\Omega} {\rm rot} ~\bb{v}_h \,\mathrm{d}x=0
\Big\}.
\end{align*}

The locking-free conforming VEM of problem \eqref{Elasticity_Eq_Traction} is
to find $\bb{u}_h \in \widehat{\bb{V}}_h$ such that
\begin{equation}
\label{Elasticity_Eq_Traction_VEM}
a_h(\bb{u}_h,\,\bb{v}_h) = \langle \bb{f}_h,\,\bb{v}_h \rangle, \quad  \bb{v}_h \in \widehat{\bb{V}}_h.
\end{equation}
Here, the construction of the computable bilinear form and the approximation of the right hand side
are the same with that in \eqref{Elasticity_Eq_VEM}.

We now establish a same result as in Lemma \ref{Lem3_Elasticity_Interp_Prop}.
\begin{lemma}
	\label{Lem4_Elasticity_Traction_Interp_Prop}
	Let $\bb{u} \in \widehat{\bb{V}}\cap\bb{H}^2(\Omega)$ be the weak solution of
	problem \eqref{Elasticity_Eq_Traction}.
	Then there exists a VEM function $\bb{u}_{\mathcal{J}} \in \widehat{\bb{V}}_h$ such that
	\begin{align*}
	\Pi_0^K{\rm div}~\bb{u}_{\mathcal{J}}  = \Pi_0^K{\rm div}~\bb{u},
	\quad K\in\mathcal{T}_h,
	\end{align*}
	and
	\begin{align*}
	|\bb{u} - \bb{u}_{\mathcal{J}}|_1 \lesssim h |\bb{u}|_2.
	\end{align*}
\end{lemma}
\begin{proof}
Let $\bb{u}_{\mathcal{I}}\in \bb{V}_h$ be the function given in Lemma \ref{Lem3_Elasticity_Interp_Prop}.
Then it immediately gives
\[
	\Pi_0^K{\rm div}~\bb{u}_{\mathcal{I}}  = \Pi_0^K{\rm div}~\bb{u}, \quad
	\Pi_0^K{\rm rot}~\bb{u}_{\mathcal{I}}  = \Pi_0^K{\rm rot}~\bb{u},
	\quad \, K\in\mathcal{T}_h.
\]
Define
\[
	\bb{u}_{\mathcal{J}} = \bb{u}_{\mathcal{I}} - \frac{1}{|\partial\Omega|}\int_{\partial\Omega} \bb{u}_{\mathcal{I}}\, \mathrm{d}s.
\]
We then have
\[
	\int_{\partial\Omega} \bb{u}_{\mathcal{J}} {\rm d}s = \bb{0},
	\quad
	\int_\Omega {\rm rot}~\bb{u}_{\mathcal{J}} \, \mathrm{d}x
	= \int_\Omega {\rm rot}~\bb{u}_{\mathcal{I}} \, \mathrm{d}x
	=\int_\Omega {\rm rot}~\bb{u}  \, \mathrm{d}x = 0 ,
\]
which indicates $\bb{u}_{\mathcal{J}} \in \widehat{\bb{V}}_h$. The estimates obviously follows from \eqref{Elasticity_Interp_uI_Est}.
\end{proof}

With the help of the above lemma and following the same argument in Section \ref{Sec4},
we are able to derive the parameter-free error estimates described as follows, whose proof is omitted for simplicity.

\begin{theorem}
	\label{Thm4_Elasticity_Eq_Traction_VEM_Est}
	Let $\bb{u}$ be the weak solution of problem \eqref{Elasticity_Eq_Traction} and
	$\bb{u}_h$ be the discrete solution of problem \eqref{Elasticity_Eq_Traction_VEM}, respectively.
	Then there holds
	\begin{align*}
	| \bb{u} - \bb{u}_h |_1 \lesssim h  \| \bb{f} \|_0.
	\end{align*}
	Moreover, if $\Omega$ is convex, we have
	\begin{align*}
	\| \bb{u} - \bb{u}_h \|_0 \lesssim h^2 \| \bb{f} \|_0.
	\end{align*}
	Here, the hidden positive constant $C$ is independent of $h$ and $\lambda$.
\end{theorem}

\section{Numerical Experiments}
\label{Sec7}
In this section, we present some numerical tests for the locking-free conforming VEMs \eqref{Elasticity_Eq_VEM}
and \eqref{Elasticity_Eq_Traction_VEM} to confirm the theoretical results. All the tests are implemented in Matlab R2016a.

Let $\bb{u}$ be the exact solution of \eqref{Elasticity_Eq} (resp. \eqref{Elasticity_Eq_Traction}) and
$\bb{u}_h$ the discrete solution of the proposed VEM \eqref{Elasticity_Eq_VEM} (resp. \eqref{Elasticity_Eq_Traction_VEM}).
For the ease of programming, we always choose $\alpha=1/2$ in our numerical method.
That means, the midpoint of each edge of the polygonal element $K$ is taken as an additional vertex of $\widetilde{K}$.
Since the VEM solution $\bb{u}_h$ is not explicitly known inside the polygonal elements,
as in \cite{Beirao-Lovadina-Russo-2017}, we will evaluate the computable errors
by comparing the exact solution $\bb{u}$ with the elliptic projection $\Pi_1^{\widetilde{K}}{\bb{u}_h}$  of $\bb{u}_h$.
In this way, the discrete $L^2$ error {\rm ErrL2} and $H^1$ error {\rm ErrH1} are respectively quantified by
\begin{align*}
{\rm ErrL2} = \| \bb{u}  - \Pi_1 \bb{u}_h \|_{0}
=\Big( \sum\limits_{\widetilde{K}\in\mathcal{T}_h^{*}} \| \bb{u}  - \Pi_1 ^{\widetilde{K}} \bb{u}_h \|_{0,\widetilde{K}}^2 \Big)^{1/2}
\quad \text{and} \quad
{\rm ErrH1} = | \bb{u}  - \Pi_1 \bb{u}_h |_{1}
=\Big( \sum\limits_{\widetilde{K}\in\mathcal{T}_h^{*}} | \bb{u}  - \Pi_1^{\widetilde{K}} \bb{u}_h |_{1,\widetilde{K}}^2 \Big)^{1/2}.
\end{align*}

\subsection{Implementation of the proposed VEM}
\label{Sec7-1}
In this subsection, to help the reader use the numerical method,
we present the computer implementation of the $L^2$ projection $\Pi_0^K{\rm div}$.
One can refer to \cite{Dassi-Vacca-2020} for the numerical realization of other standard terms.
Let $\phi^{\intercal} = (\phi_1, \cdots ,\phi_N)$ be the basis functions of $V_1(\widetilde{K})$ and denote
\[
	\bb{\phi}^\intercal = ( {\overline \phi_1}, \cdots ,{\overline \phi_N},{\underline \phi_1}, \cdots ,{\underline \phi_N} ),
	\quad N = 2N_K
\]
to be the basis functions of the vector space $\bb{V}_1(\widetilde{K})$, where
\[
	{\overline \phi_i} =
	\begin{pmatrix}
	\phi_i \\
	0
	\end{pmatrix},
	\quad ~~
	{\underline \phi_i} =
	\begin{pmatrix}
	0 \\
	\phi_i
	\end{pmatrix},
	\quad ~i = 1, 2, \cdots ,{N}.
\]
Then the definition of  $L^2$ projection $\Pi_0^K{\rm div}$ can be written in vector form as
\[
	\int_K \Pi_0^K{\rm div}~{\bb \phi^\intercal} \,\mathrm{d}x = \int_K {\rm div}~{\bb \phi^\intercal}\, \mathrm{d}x \qquad
	\mbox{or} \qquad \bb{H}_0 \bb{\Pi}_{0*} = \bb{C}_0,
\]
where
\[
	\Pi_0^K{\rm div}~{\bb \phi^\intercal} = \bb{\Pi}_{0*}, \quad
	\bb{H}_0 = |K|, \quad  \bb{C}_0 = \int_K {\rm div}~{\bb \phi^\intercal} \, \mathrm{d}x.
\]
By integration by parts,
\[
	\bb{C}_0 = \int_K {\rm div}~{\bb \phi^{\intercal}} \,\mathrm{d}x
	= \int_{\partial {K}} \bb{\phi}^{\intercal}  \cdot \bb{n}_K \, \mathrm{d}s
	= \int_{\partial {K}} [\phi^{\intercal}   \cdot {n_1},\phi^{\intercal} \cdot {n_2}] \,\mathrm{d}s.
\]
Owing to the Kronecker property $\chi_i(\phi_j) = \delta_{ij}$, the two integrals of the right-hand side can be easily computed
by using the trapezoidal rule and the associated local stiffness matrix
\[
	\bb{B}_K = \int_K \Pi_0^K{\rm div}~\bb{\phi} \Pi_0^K{\rm div}~ \bb{\phi}^{\intercal} \, \mathrm{d}x
	= \bb{\Pi}_{0*}^\intercal \bb{H}_0 \bb{\Pi}_{0*}
\]
will be elementwise assembled.

On the other hand, for the pure traction problem, the discrete admissible function must satisfy the constraints
\[
	\int_{\partial\Omega} \bb{u}_h \, \mathrm{d}s = \bb{0} \quad \mbox{and} \quad
	\int_\Omega {\rm rot}~\bb{u}_h \, \mathrm{d}x = 0.
\]
To impose these constraints, we introduce Lagrange multipliers $\beta_i \in \mathbb{R}$ and consider the augmented variational formulation:
Find $(\boldsymbol{u}_h,\beta_1, \beta_2,\beta_3)\in \overline{\bb{V}}_h \times \mathbb{R}\times \mathbb{R}\times \mathbb{R}$ such that
\begin{equation}
\label{Elasticity_augVEM}
\begin{cases}
	a_h(\bb{u}_h,\bb{v}_h) + \beta_1 \int_{\partial\Omega} \bb{u}_{h,1} \, \mathrm{d}s
		+ \beta_2\int_{\partial\Omega} \bb{u}_{h,2} \, \mathrm{d}s
		+ \beta_3 \int_\Omega {\rm rot}~\bb{u}_h \, \mathrm{d}x
	& = \langle \bb{f}_h, \, \bb{v}_h \rangle, \qquad \bb{v}_h\in \overline{\bb{V}}_h, \\
	\mu_1\int_{\partial\Omega} \bb{v}_{h,1} \, \mathrm{d}s & = 0, \qquad \mu_1 \in \mathbb{R},\\
	\mu_2\int_{\partial\Omega} \bb{v}_{h,2} \, \mathrm{d}s & = 0, \qquad \mu_2 \in \mathbb{R},\\
	\mu_3\int_\Omega {\rm rot}~\bb{v}_h  \, \mathrm{d}x & = 0, \qquad \mu_3 \in \mathbb{R},\\
\end{cases}
\end{equation}
where $\bb{u}_h = (\bb{u}_{h,1}, \bb{u}_{h,2})^{\intercal}$, $\bb{v}_h = (\bb{v}_{h,1}, \bb{v}_{h,2})^{\intercal}$
and
\begin{align*}
\overline{\bb{V}}_h = \Big\{
\bb{v}_h\in\bb{H}^1(\Omega) ~|~ \bb{v}_h|_{\widetilde{K}} \in \widehat{\bb{V}}_1(\widetilde{K}),
\quad \widetilde{K}\in\mathcal{T}_h
\Big\}.
\end{align*}
Let $\bb{\varphi}_i$, $i=1,\cdots, N$ be the nodal basis function of $\overline{\bb{V}}_h$,
where $N$ is the dimension of $\overline{\bb{V}}_h$.
Then we can write
\[
	\bb{u}_h = \sum\limits_{i=1}^N \chi_i(\bb{u}_h)\bb{\varphi}_i
	=: \bb{\varphi}^{\intercal} \bb{\chi}(\bb{u}_h).
\]
Plug the above expression into \eqref{Elasticity_augVEM}, and take $\bb{v}_h = \bb{\varphi}_j$ for $j=1,\cdots,N$.
We have
\[
	\begin{cases}
		\sum\limits_{i=1}^Na_h(\bb{\varphi}_i,\bb{\varphi}_j) \, \chi_i
		 + \beta_1\int_{\partial\Omega} \bb{\varphi}_{i,1} \, \mathrm{d}s
		 + \beta_2\int_{\partial\Omega} \bb{\varphi}_{i,2} \, \mathrm{d}s
		 + \beta_3 \int_\Omega {\rm rot}~\bb{\varphi}_i \, \mathrm{d}x & = 0, \quad j=1,2,\cdots,N, \\
		\sum\limits_{i=1}^N \int_{\partial\Omega} \bb{\varphi}_{i,1} \, \mathrm{d}s \,\chi_i & = 0, \\
		\sum\limits_{i=1}^N \int_{\partial\Omega} \bb{\varphi}_{i,2}  \, \mathrm{d}s \, \chi_i & = 0, \\
		\sum\limits_{i=1}^N \int_\Omega {\rm rot}~\bb{\varphi}_i \, \mathrm{d}x \, \chi_i & = 0 .
	\end{cases}
\]
Let
\[
	\bb{A} = \Big(a_h(\bb{\varphi}_j,\bb{\varphi}_i)\Big)_{N\times N}, \quad
	\bb{d}_1 = \Big(\int_{\partial\Omega} \bb{\varphi}_{i,1} \, \mathrm{d}s \Big)_{N\times 1}, \quad
	\bb{d}_2 = \Big(\int_{\partial\Omega} \bb{\varphi}_{i,2} \, \mathrm{d}s \Big)_{N\times 1}, \quad
	\bb{d}_3 = \Big(\int_\Omega {\rm rot}~\bb{\varphi}_i \,\mathrm{d}x \Big)_{N\times 1}
\]
and
\[
	\bb{f} = \Big(  \langle \bb{f}_h, \bb{\varphi}_i \rangle \Big)_{N\times 1}.
\]
The linear system can be written in matrix form:
\[
	\begin{bmatrix}
	\bb{A}              & \bb{d}_1 & \bb{d}_2  & \bb{d}_3 \\
	\bb{d}_1^{\intercal}  &     0    &           & \\
	\bb{d}_2^{\intercal}  &          &     0     & \\
	\bb{d}_3^{\intercal}  &          &           & 0\\
	\end{bmatrix}
	\begin{pmatrix}
	\bb{\chi} \\
	\beta_1 \\
	\beta_2 \\
	\beta_3
	\end{pmatrix}
	=\begin{pmatrix}
	\bb{f} \\
	0 \\
	0 \\
	0
	\end{pmatrix}.
\]

\subsection{The pure displacement problem}
\label{Sec7-2}

{\bf Test 1:} We consider  the linear elasticity problem \eqref{Elasticity_Eq} on $\Omega=(0,1)^2$ with homogeneous Dirichlet boundary conditions.
The load term $\bb{f}$ is chosen such that the exact solution of \eqref{Elasticity_Eq} is
\begin{align*}
\bb{u}(x_1,\,x_2) =
\begin{pmatrix}
(-1+\cos{(2\pi{x_1})})\sin{(2\pi{x_2})} \\
(1-\cos{(2\pi{x_2})})\sin{(2\pi{x_1})}
\end{pmatrix}
+\frac{1}{\mu+\lambda}\sin{(\pi{x_1})}\sin{({\pi{x_2}})}
\begin{pmatrix}
1 \\
1
\end{pmatrix}.
\end{align*}

\begin{figure}[H]
	\centering
	\subfigure[Unstructured triangles]
	{\includegraphics[scale=0.26]{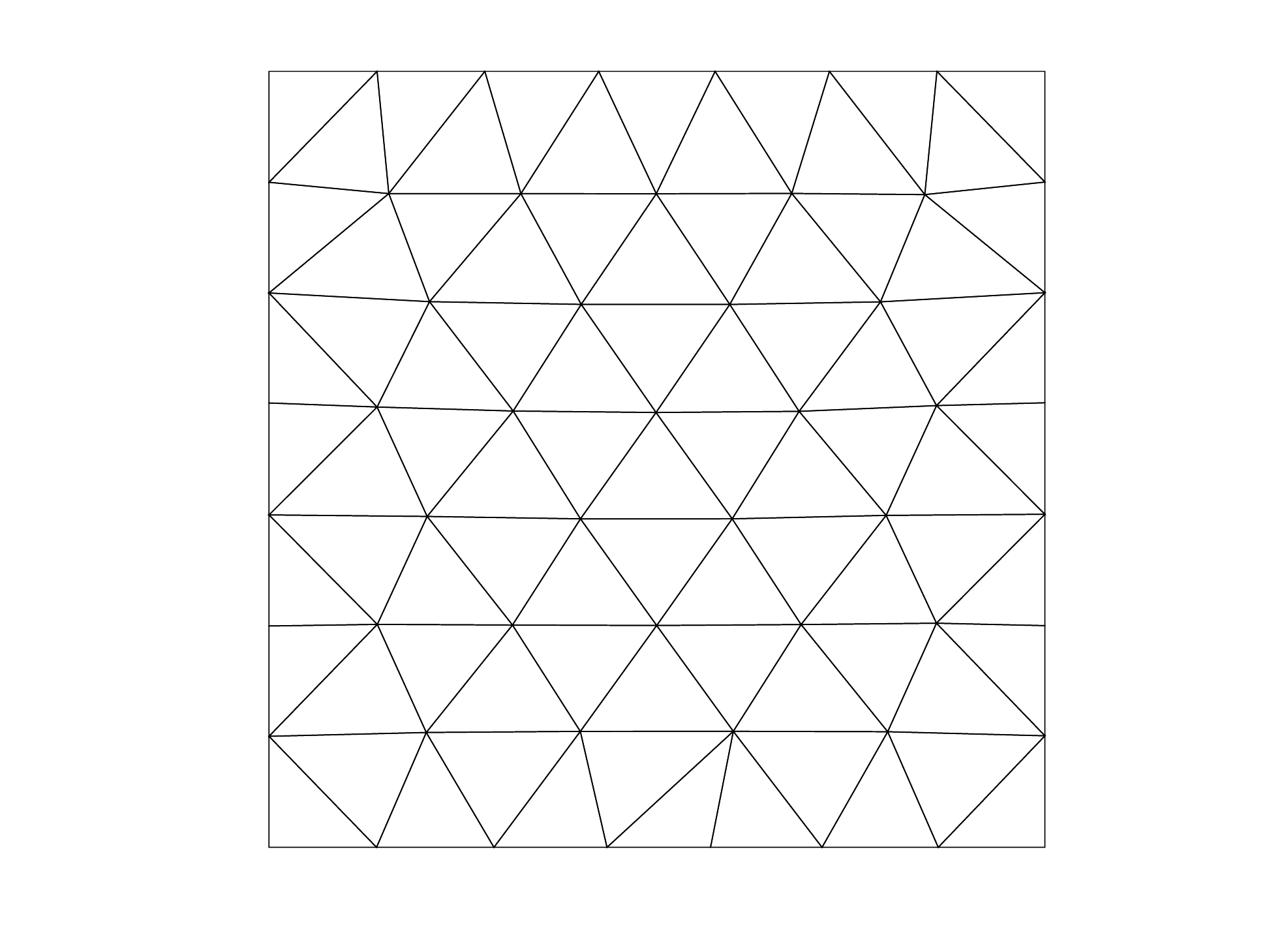}}
	\subfigure[Distorted squares]
	{\includegraphics[scale=0.26]{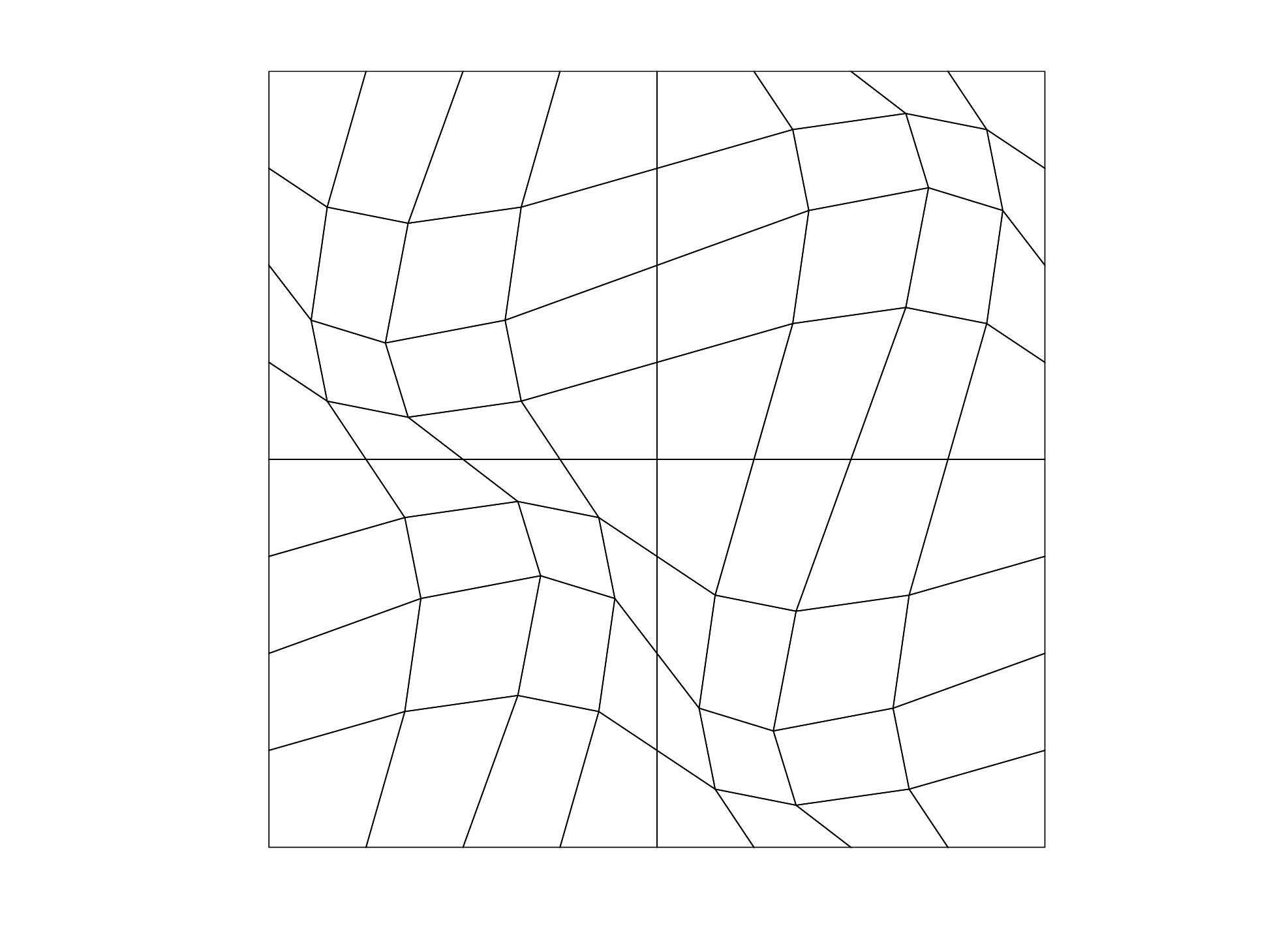}}
	\subfigure[Voronoi polygons]
	{\includegraphics[scale=0.26]{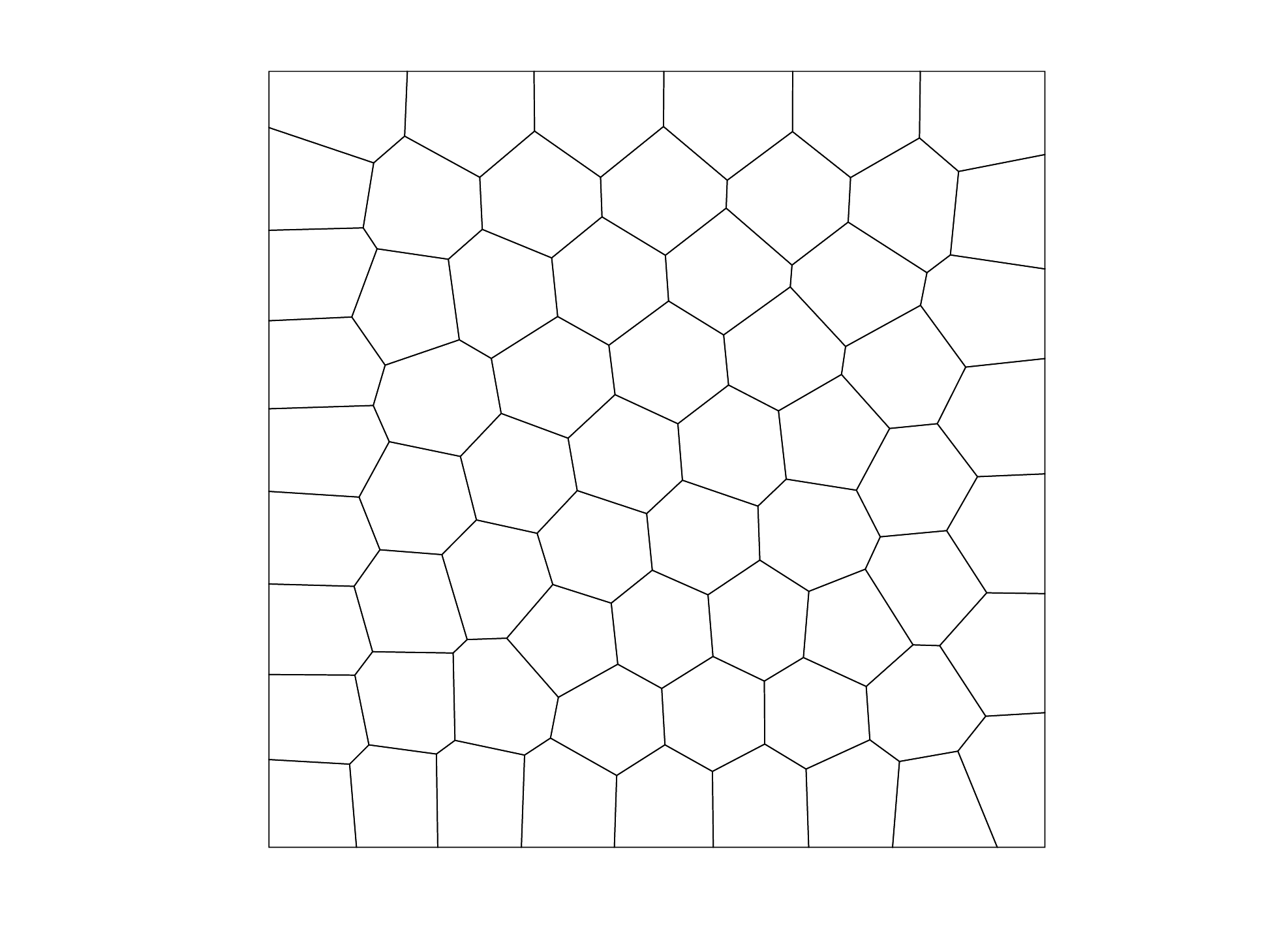}}
	\caption{Three families of meshes in Test 1.} 	
	\label{Fig4_Test1a_Meshes}
\end{figure}

{\bf Case 1:} We first consider the convergence with respect to different mesh shapes.
To do so, we choose three different sequences of meshes, labeled by $\mathcal{M}_1$, $\mathcal{M}_2$, $\mathcal{M}_3$,
as shown in Fig. \ref{Fig4_Test1a_Meshes}, where
$\mathcal{M}_1$ is a sequence of meshes composed of unstructured triangles,
$\mathcal{M}_2$ is a sequence of meshes composed of distorted squares,
and $\mathcal{M}_3$ is a sequence of meshes composed of Voronoi polygons.
The distorted meshes are obtained by mapping the position $(\xi,\,\zeta)$ of equal squares meshes
through the smooth coordinate transformation
\begin{align*}
x_1 = \xi + \frac{1}{10}\sin{(2\pi\xi)}\sin{(2\pi\zeta)},
\quad
x_2 = \zeta + \frac{1}{10}\sin{(2\pi\xi)}\sin{(2\pi\zeta)}.
\end{align*}
And the polygonal meshes are generated by the MATLAB toolbox - PolyMesher introduced in \cite{Talischi-Paulino-Pereira-2012}.

Let $\mu=1$ and $\lambda=10^6$. We display the nodal values of the elliptic projection $\Pi_1\bb{u}_h$ on $\mathcal{M}_3$
in Fig. \ref{Fig5_Test1a_SolutionInfo}, from which we observe that the numerical solutions are well matched with the exact ones.

\begin{figure}[H]
	\centering
	\includegraphics[scale=0.5]{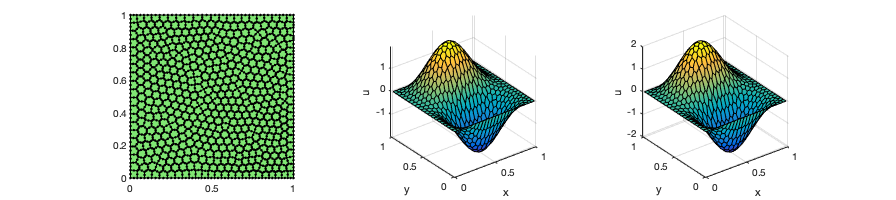} 
	\caption{Numerical and exact solutions for $\lambda=10^6$ on a polygonal mesh of $\mathcal{M}_3$ in Test 1.
	(left: Mesh512, middle: numerical solution, right: exact solution)
	}
	\label{Fig5_Test1a_SolutionInfo}
\end{figure}

The $H^1$ and $L^2$ errors computed as functions of mesh size $h$ for the mesh sequences $\{\mathcal{M}_i\}_{i=1}^3$
are depicted in Fig. \ref{Fig6_Test1a_ConvRates} as a log-log plot.
For each fixed pair $(\lambda, \mu)$, the convergence rate with respect to $h$ is estimated by assuming ${\rm Err}(h)=ch^{\alpha}$
and  computing a least squares fit to this log-linear relation.
We see that the errors {\rm ErrH1} and {\rm ErrL2} converge at the optimal rate $\mathcal{O}(h^1)$ and $\mathcal{O}(h^2)$
respectively, and are not affected by the shape of the mesh discretization.
It is worth noting that, the proposed VEM shows a satisfactory stability with respect to the nearly incompressible case
($\lambda=10^6$).

\begin{figure}[!htb]
	\centering
	\subfigure[Unstructued triangles]
	{\includegraphics[scale=0.26]{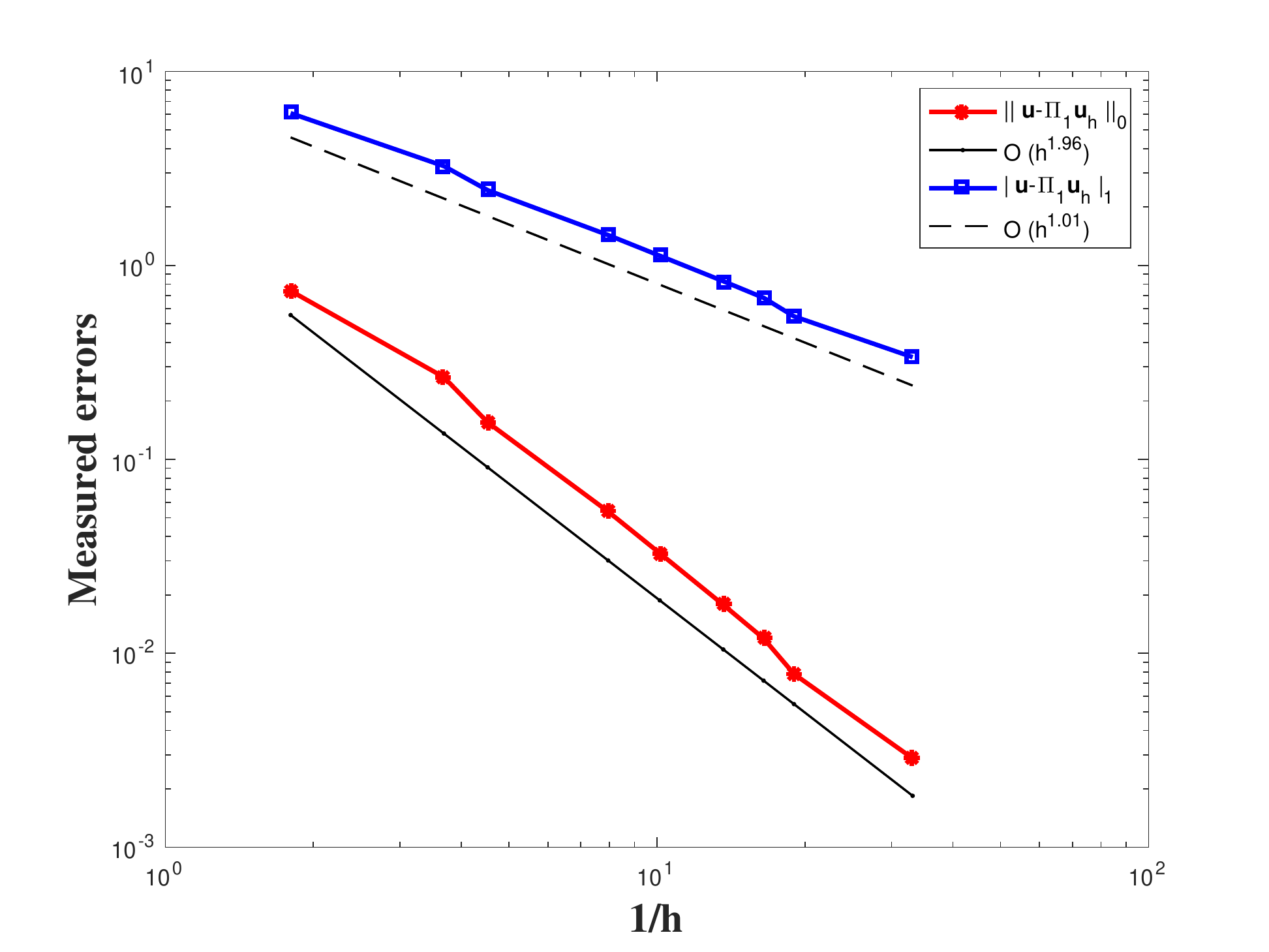}}
	\subfigure[Distorted squares]
	{\includegraphics[scale=0.26]{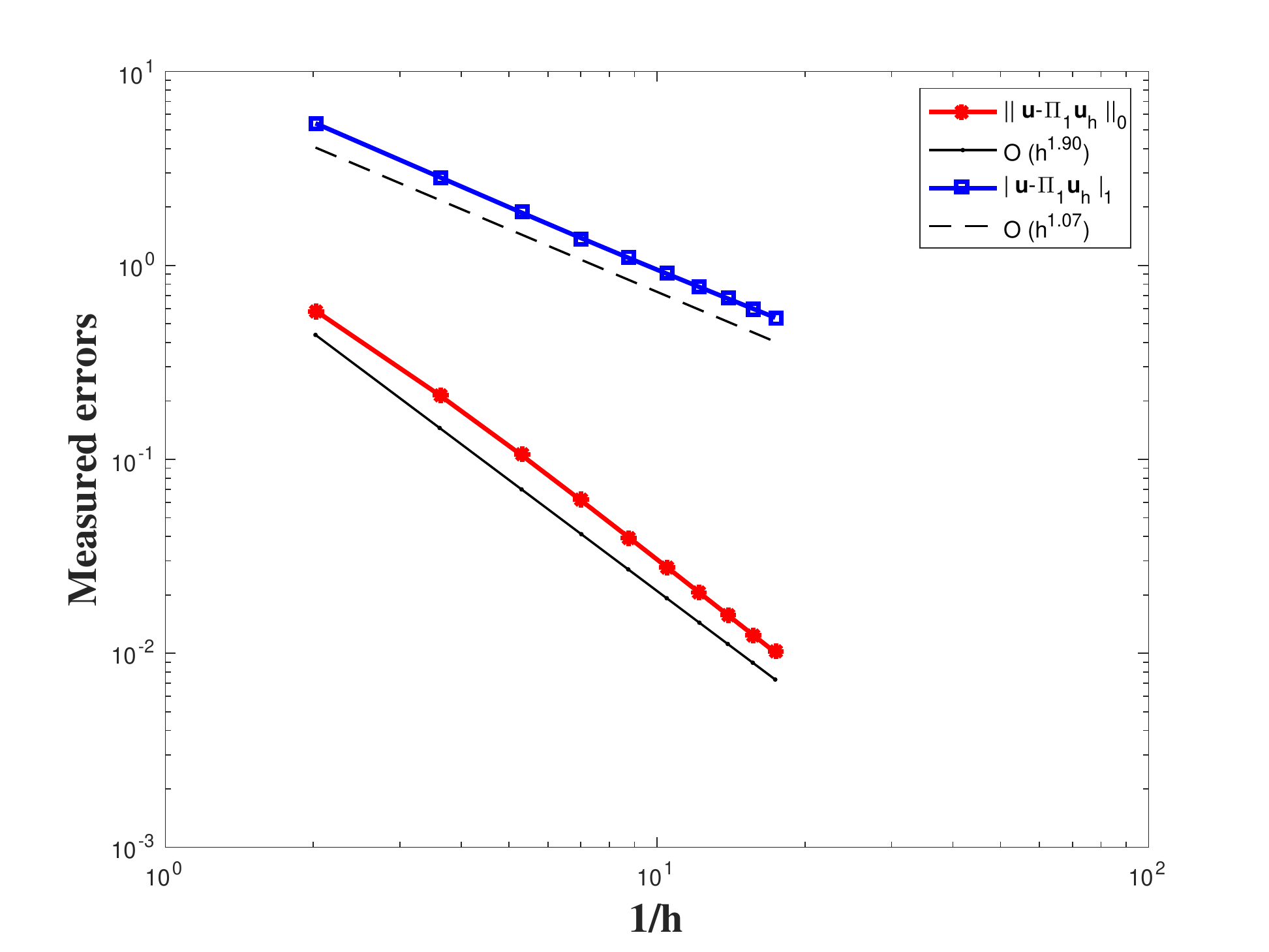}}
	\subfigure[Voronoi polygons]
	{\includegraphics[scale=0.26]{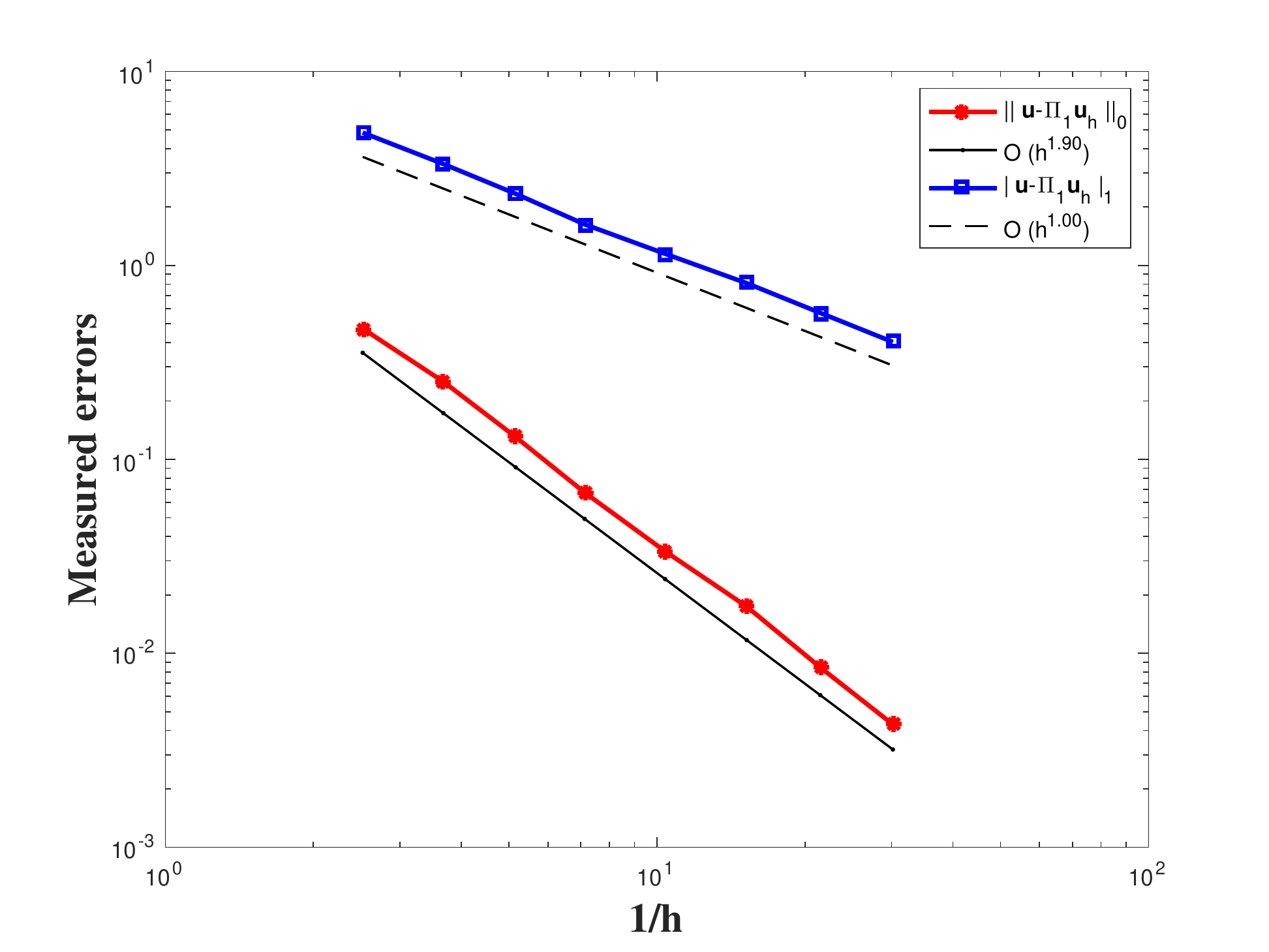}}
	\caption{$L^2$ and $H^1$ errors for $\lambda=10^6$ on three different sequences of meshes in Test 1.
		}
	\label{Fig6_Test1a_ConvRates}
\end{figure}

\par

{\bf Case 2:} 
We next investigate the locking-free property of the proposed VEM, i.e., the robustness with respect to the choice of $\lambda$.
In this case, we only focus on considering the Voronoi polygonal meshes $\mathcal{M}_3$.

Let $\mu=1$ and $\lambda\in\{1,\,10^4,\,10^7\}$. The convergence graphs of the errors {\rm ErrH1} and {\rm ErrL2}
for different values of $\lambda$ are displayed in Fig. \ref{Fig7_Test1b_Lambda_ConvRates},
and the corresponding errors {\rm ErrH1} and {\rm ErrL2} are shown in Tab. \ref{Tab1_Errors_Lambda_Choices_1e047}.
We observe that all the errors are almost the same and the method converges with the optimal rates,
which is in well agreement with the theoretical predictions in Theorems \ref{Thm2_Elasticity_Eq_VEM_H1_Est}
and \ref{Thm3_Elasticity_Eq_VEM_L2_Est}.
\begin{table}[!ht]
	\small
	\setlength{\belowcaptionskip}{6pt}
	\renewcommand{\arraystretch}{1.0}
	\centering	
	\caption{$H^1$ and $L^2$ errors for $\lambda\in\{1,\,10^4,\,10^7\}$ on polygonal meshes $\mathcal{M}_3$ in Test 1.}
	\label{Tab1_Errors_Lambda_Choices_1e047}
	\begin{tabular}{cccccccc}\hline
		Mesh
		& Choices of $\lambda$
		&\multicolumn{2}{c}{$\lambda=1$}
		&\multicolumn{2}{c}{$\lambda=10^4$}
		&\multicolumn{2}{c}{$\lambda=10^7$}
		\\ 
		h & {\rm $\#$DOFs}
		& {\rm ErrH1}  & {\rm ErrL2} & {\rm ErrH1}  & {\rm ErrL2} & {\rm ErrH1}  & {\rm ErrL2}
		\\ \hline
		0.396819 & 166
		& 4.859e+0 & 4.708e-1 & 4.833e+0 & 4.722e-1 &  4.833e+0 & 4.722e-1
		\\
		0.272025 & 326
		& 3.345e+0 & 2.494e-1 & 3.331e+0 & 2.515e-1 & 3.331e+0 & 2.515e-1
		\\
		0.193715 & 646
		& 2.352e+0 & 1.302e-1 & 2.337e+0 & 1.302e-1 & 2.337e+0 & 1.302e-1
		\\
		0.140331 & 1278
		& 1.635e+0 & 6.699e-2 & 1.628e+0 & 6.748e-2 & 1.628e+0 & 6.748e-2
		\\
		0.096262 & 2530
		& 1.155e+0 & 3.353e-2 & 1.147e+0 & 3.357e-2 & 1.147e+0 & 3.357e-2
		\\
		0.065690 & 5066
		& 8.148e-1 & 1.742e-2 & 8.092e-1 & 1.743e-2 & 8.092e-1 & 1.743e-2
		\\
		0.046554 & 10186
		& 5.719e-1 & 8.468e-3 & 5.690e-1 & 8.473e-3 & 5.690e-1 & 8.473e-3
		\\
		0.033120 & 20326
		& 4.047e-1 & 4.306e-3 & 4.025e-1 & 4.300e-3 & 4.025e-1 & 4.300e-3
		\\ \hline
	\end{tabular}
\end{table}

\begin{figure}[H]
	\centering
	\subfigure[$\lambda=1$]
	{\includegraphics[scale=0.26]{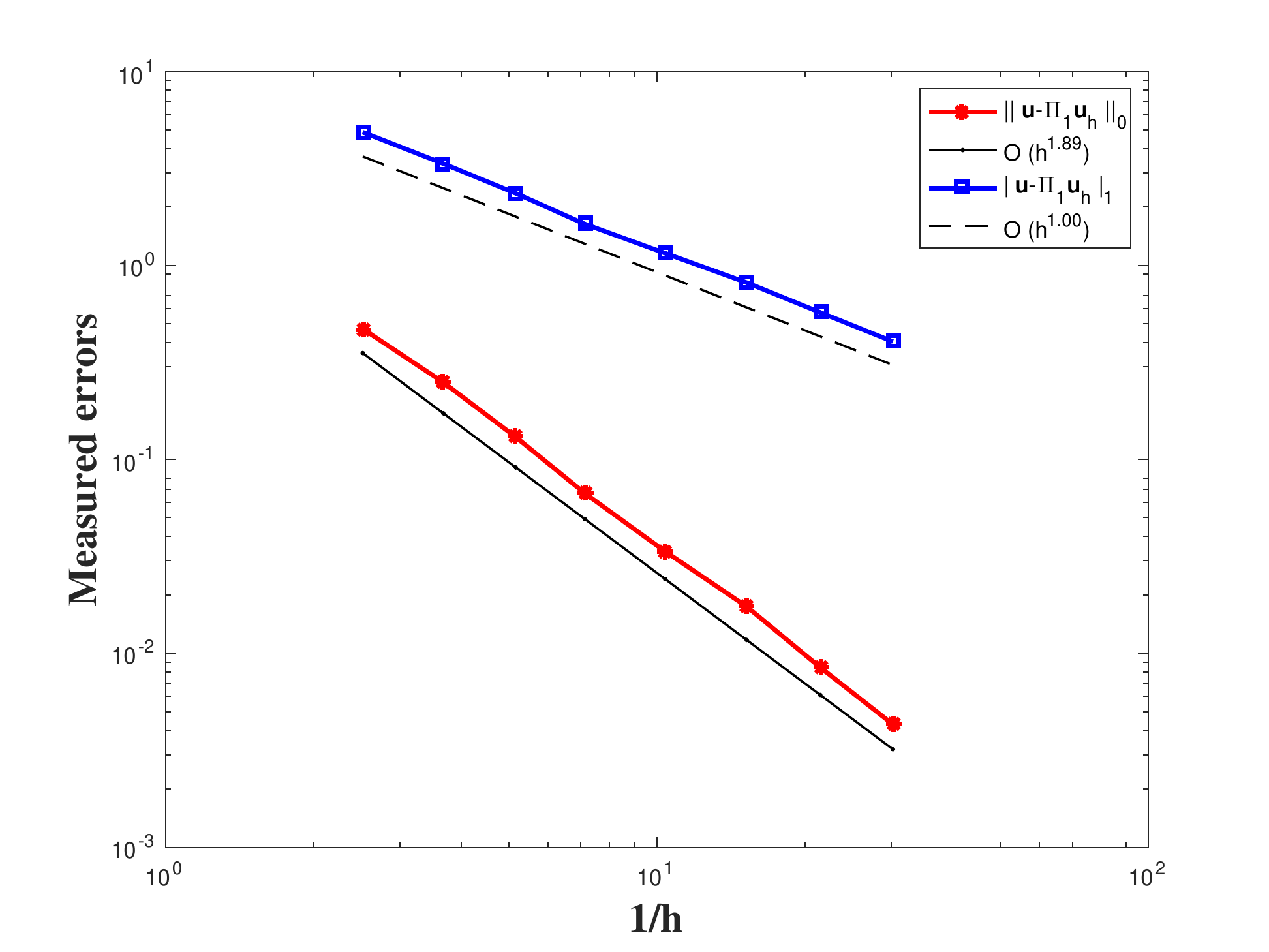}}
	\subfigure[$\lambda=10^4$]
	{\includegraphics[scale=0.26]{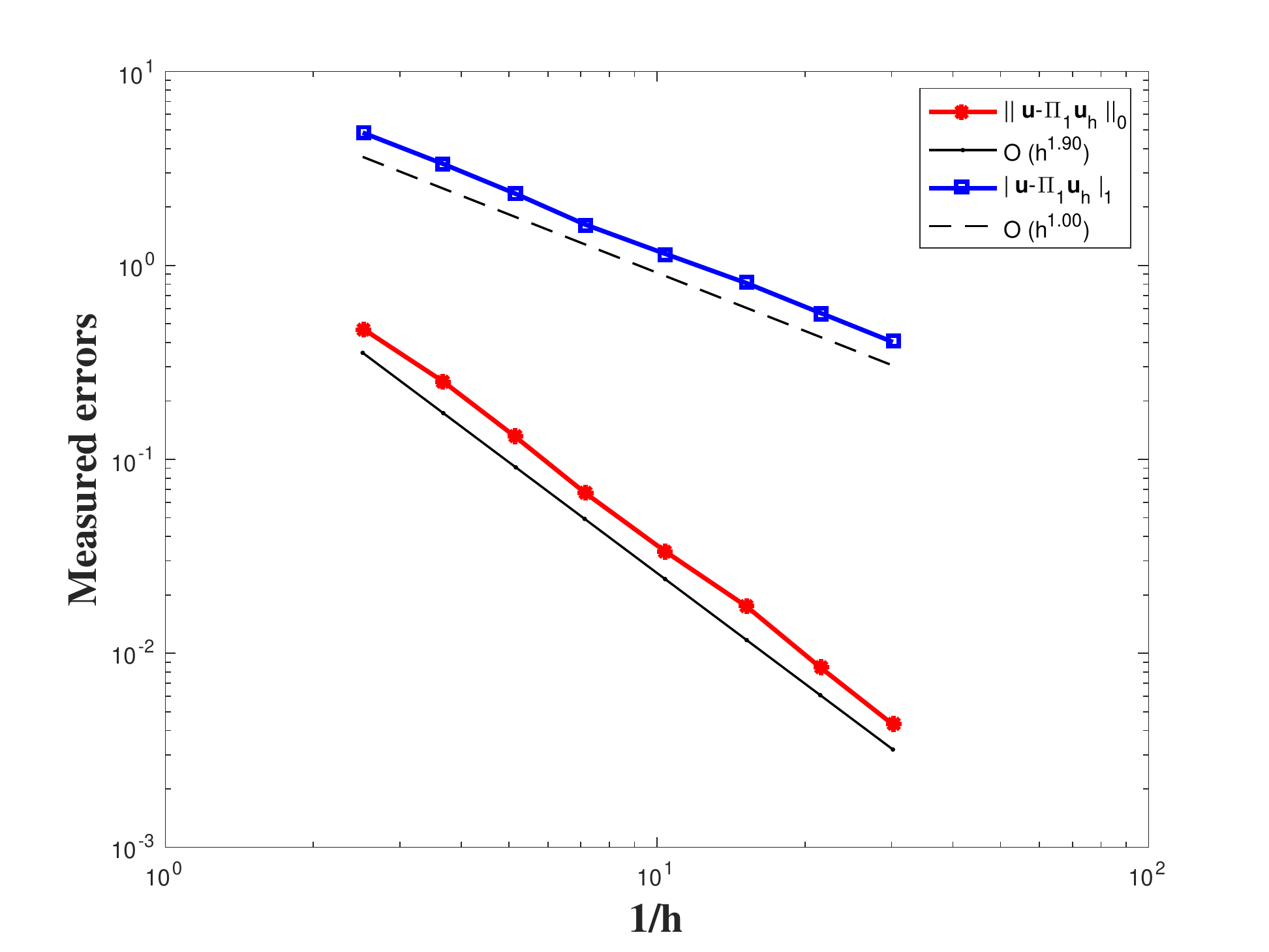}}
	\subfigure[$\lambda=10^7$]
	{\includegraphics[scale=0.26]{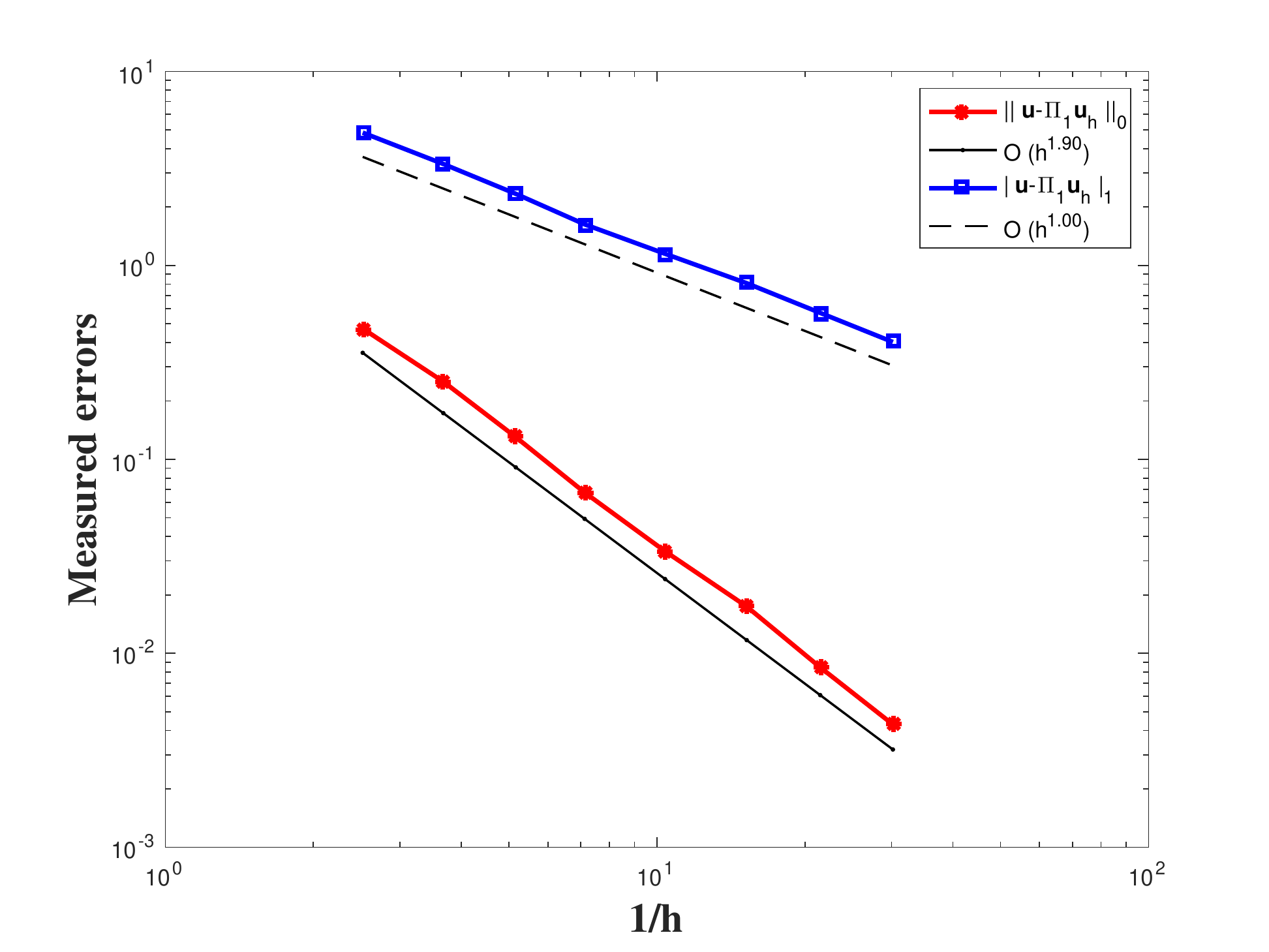}}
	\caption{$L^2$ and $H^1$ errors for $\lambda\in\{1,\,10^4,\,10^7\}$ on polygonal meshes $\mathcal{M}_3$ in Test 1.
	}
	\label{Fig7_Test1b_Lambda_ConvRates}
\end{figure}

Furthermore, we take $\lambda\in\{10^j,j=0,1,\cdots,7\}$ on a fixed mesh in $\mathcal{M}_3$ with $h=0.0656898$.
As shown in Tab. \ref{Tab2_Errors_Lambda_Mesh512},
the errors {\rm ErrH1} and {\rm ErrL2} are insensitive to the choice of $\lambda$,
thus the proposed VEM is robust with respect to $\lambda$, i.e., it is volume locking-free.
\begin{table}[!ht]
	\setlength{\belowcaptionskip}{6pt}
	\renewcommand{\arraystretch}{1.0}
	\centering	
	\caption{$H^1$ and $L^2$ errors for different values of $\lambda$ on one fixed mesh of $\mathcal{M}_3$ in Test 1.}
	\label{Tab2_Errors_Lambda_Mesh512}
	\begin{tabular}{ccccccccc}
		\hline
		${\rm Errors}\setminus\lambda$
		& $1$ & $10$ & $10^2$ & $10^3$ & $10^4$ & $10^5$ & $10^6$ & $10^7$
		\\
		\hline
		{\rm ErrH1}	
		& 8.148e-1  & 8.096e-1  & 8.092e-1
		& 8.092e-1  & 8.092e-1  & 8.092e-1
		& 8.092e-1  & 8.092e-1
		\\
		{\rm ErrL2}	
		& 1.742e-2  & 1.741e-2  & 1.743e-2
		& 1.743e-2  & 1.743e-2  & 1.743e-2
		& 1.743e-2  & 1.743e-2
		\\
		\hline
	\end{tabular}
\end{table}

\subsection{The mixed boundary conditions}
\label{Sec7-3}

{\bf Test 2:} We solve the following linear elasticity problem on $\Omega=(0,1)^2$ with general boundary conditions:
\begin{equation}
\label{Elasticity_Eq_Gen_Bds}
\begin{cases}
- {\mathbf{div}}~\bb{\sigma}(\bb{u}) = \bb{f} & \text{in} ~~\Omega, \\
\bb{u}  = \bb{g}_1  & \text{on} ~~\Gamma_1. \\
\bb{\sigma}(\bb{u}) \bb{n}  = \bb{g}_2  & \text{on} ~~\Gamma_2,
\end{cases}
\end{equation}
where, $\Gamma_2=\big\{(x_1,\,x_2) ~|~ 0\leq{x}_1\leq1, x_2=0\big\}$ and $\Gamma_1=\partial\Omega\backslash\Gamma_2$.
We select the load term $\bb{f}$ and the boundary conditions such that the analytical solution of problem
\eqref{Elasticity_Eq_Gen_Bds} (cf. \cite{Wang-Wang-Wang-Zhang-2016}) is
\begin{align*}
\bb{u}(x_1,\,x_2) =
\begin{pmatrix}
\sin{(x_1)}\sin{(x_2)} \\
\cos{(x_1)}\cos{(x_2)}
\end{pmatrix}
+\frac{1}{\lambda}
\begin{pmatrix}
x_1 \\
x_2
\end{pmatrix}.
\end{align*}
Here the first term is divergence-free, but the second term provides a constant divergence $(2/\lambda)$.
It is clear that ${\rm div~}\boldsymbol{u}=0$ as $\lambda$ tends to infinity.

We take $\mu=0.5$ and $\lambda\in\{1,\,10^4,\,10^7,10^{10}\}$ on Voronoi polygonal meshes $\mathcal{M}_3$.
The nodal values of the elliptic projection $\Pi_1\bb{u}_h$ for $\lambda=10^{10}$ are shown in Fig. \ref{Fig8_Test2a_SolutionInfo}.
The results and convergence graph of the two errors {\rm ErrH1} and {\rm ErrL2} versus the mesh size $h$ for different choice of $\lambda$ are illustrated
in Tabs. \ref{Tab3_Test2a_Lambda_ErrH1}-\ref{Tab4_Test2a_Lambda_ErrL2} and Fig. \ref{Fig9_Test2a_Lambda_ConvRates}, respectively.
As observed, the optimal rates of convergence are achieved for both cases.

\begin{figure}[H]
	\centering
	\includegraphics[scale=0.5]{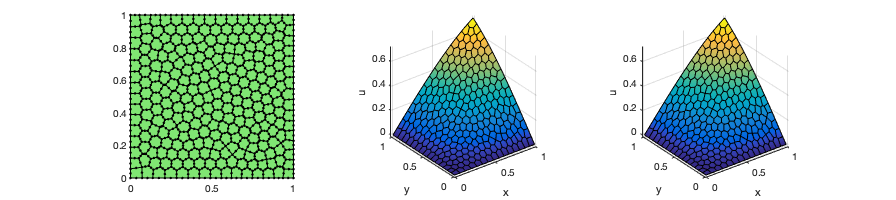}
	\caption{Numerical and exact solutions for $\lambda=10^{10}$ on a polygonal mesh of $\mathcal{M}_3$ in Test 2.
	(left: Mesh512, middle: numerical solution, right: exact solution)
	}
	\label{Fig8_Test2a_SolutionInfo}
\end{figure}

\begin{table}[H]
	\small
	\setlength{\belowcaptionskip}{6pt}
	\renewcommand{\arraystretch}{0.9}
	\centering	
	\caption{Convergence rate w.r.t. $H^1$ norm for $\lambda\in\{1,\,10^4,\,10^7,10^{10}\}$ on meshes $\mathcal{M}_3$ in Test 2.}
	\label{Tab3_Test2a_Lambda_ErrH1}
	\begin{tabular}{cccccccccc}
		\hline
		$\lambda\setminus{h}$
		& 0.396819  & 0.275890 & 0.193715 & 0.140331 & 0.096262 & 0.065690 & 0.046554 & 0.033120 & {\rm Rate}
		\\
		\hline
		$1$	
		& 1.076e-1  & 7.381e-2  & 5.238e-2
		& 3.727e-2  & 2.605e-2  & 1.847e-2
		& 1.297e-2  & 9.158e-3  & 0.99
		\\ 
		$10^4$	
		& 1.079e-1  & 7.386e-2  & 5.235e-2
		& 3.726e-2  & 2.604e-2  & 1.846e-2
		& 1.297e-2  & 9.155e-3  & 0.99
		\\ 
		$10^7$	
		& 1.079e-1  & 7.386e-2  & 5.235e-2
		& 3.726e-2  & 2.604e-2  & 1.846e-2
		& 1.297e-2  & 9.155e-3  & 0.99
		\\ 
		$10^{10}$	
		& 1.079e-1  & 7.386e-2  & 5.235e-2
		& 3.726e-2  & 2.604e-2  & 1.846e-2
		& 1.297e-2  & 9.155e-3  & 0.99
		\\
		\hline
	\end{tabular}
\end{table}

\begin{table}[H]
	\small
	\setlength{\belowcaptionskip}{6pt}
	\renewcommand{\arraystretch}{0.9}
	\centering	
	\caption{Convergence rate w.r.t. $L^2$ norm for $\lambda\in\{1,\,10^4,\,10^7,10^{10}\}$ on meshes $\mathcal{M}_3$ in Test 2.}
	\label{Tab4_Test2a_Lambda_ErrL2}
	\begin{tabular}{cccccccccc}
		\hline
		$\lambda\setminus{h}$
		& 0.396819  & 0.275890 & 0.193715 & 0.140331 & 0.096262 & 0.065690 & 0.046554 & 0.033120 & {\rm Rate}
		\\
		\hline
		$1$	
		& 6.188e-3  & 2.739e-3  & 1.311e-3
		& 6.473e-4  & 3.094e-4  & 1.518e-4
		& 7.392e-5  & 3.681e-5  & 2.05
		\\ 
		$10^4$	
		& 6.113e-3  & 2.759e-3  & 1.350e-3
		& 6.775e-4  & 3.305e-4  & 1.617e-4
		& 8.005e-5  & 4.015e-5  & 2.01
		\\ 
		$10^7$	
		& 6.113e-3  & 2.759e-3  & 1.350e-3
		& 6.775e-4  & 3.305e-4  & 1.617e-4
		& 8.005e-5  & 4.015e-5  & 2.01
		\\ 
		$10^{10}$	
		& 6.113e-3  & 2.759e-3  & 1.350e-3
		& 6.775e-4  & 3.305e-4  & 1.618e-4
		& 8.015e-5  & 4.021e-5  & 2.01
		\\
		\hline
	\end{tabular}
\end{table}

\begin{figure}[H]
	\centering
		\subfigure[$\lambda=1$]
	{\includegraphics[scale=0.37]{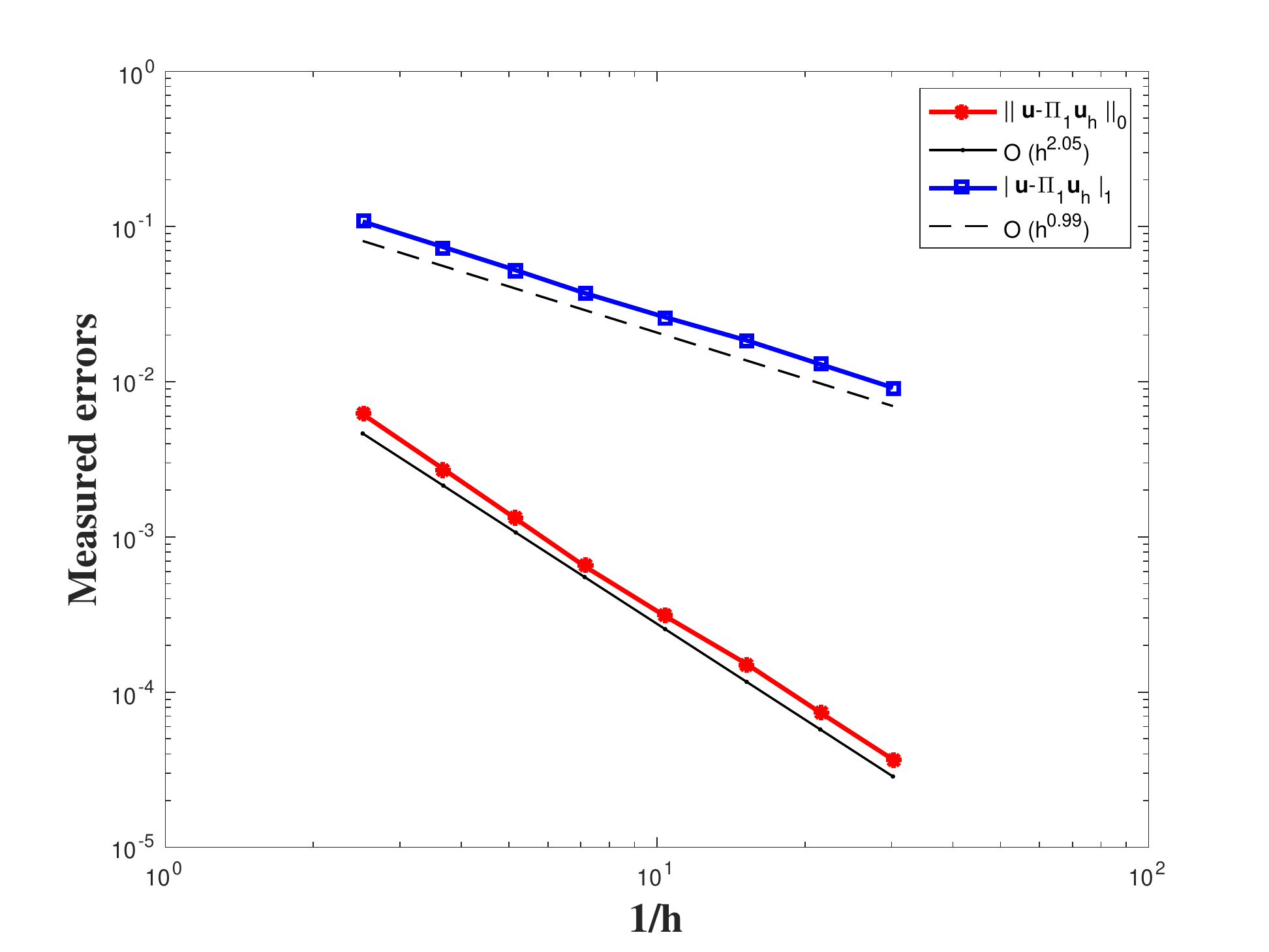}}
	\subfigure[$\lambda=10^4$]
	{\includegraphics[scale=0.37]{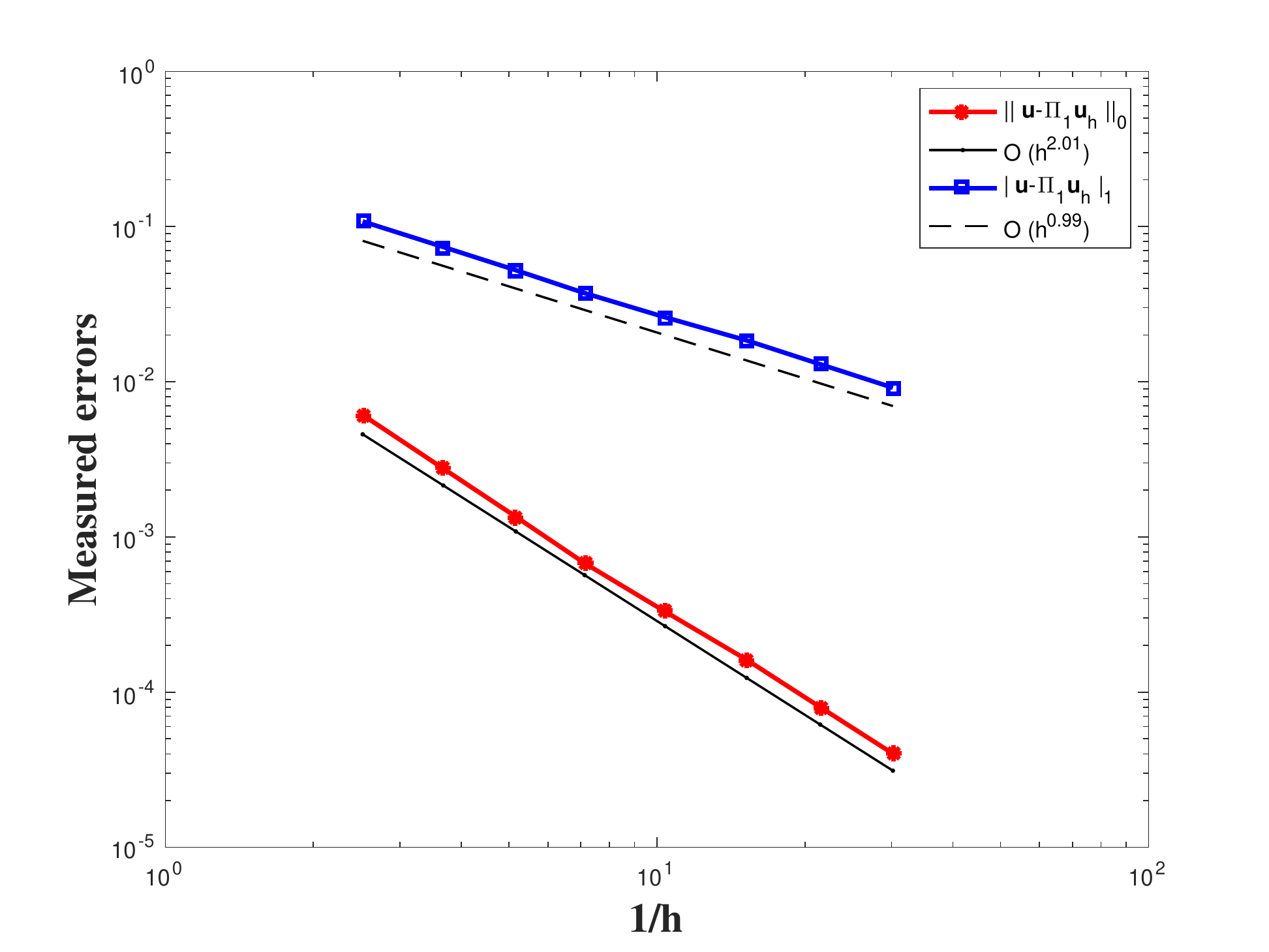}} \\
	\subfigure[$\lambda=10^7$]
	{\includegraphics[scale=0.37]{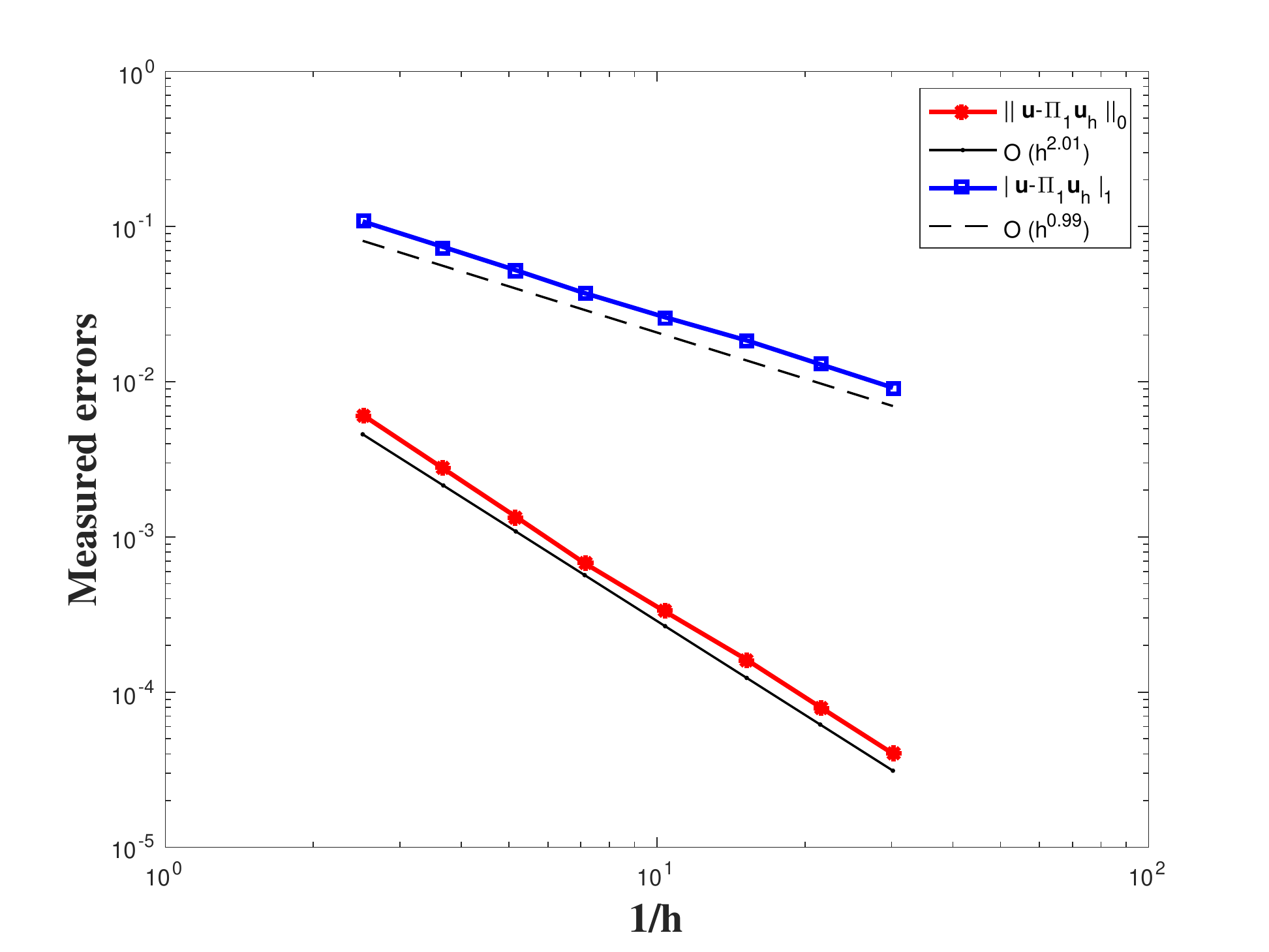}}
	\subfigure[$\lambda=10^{10}$]
	{\includegraphics[scale=0.37]{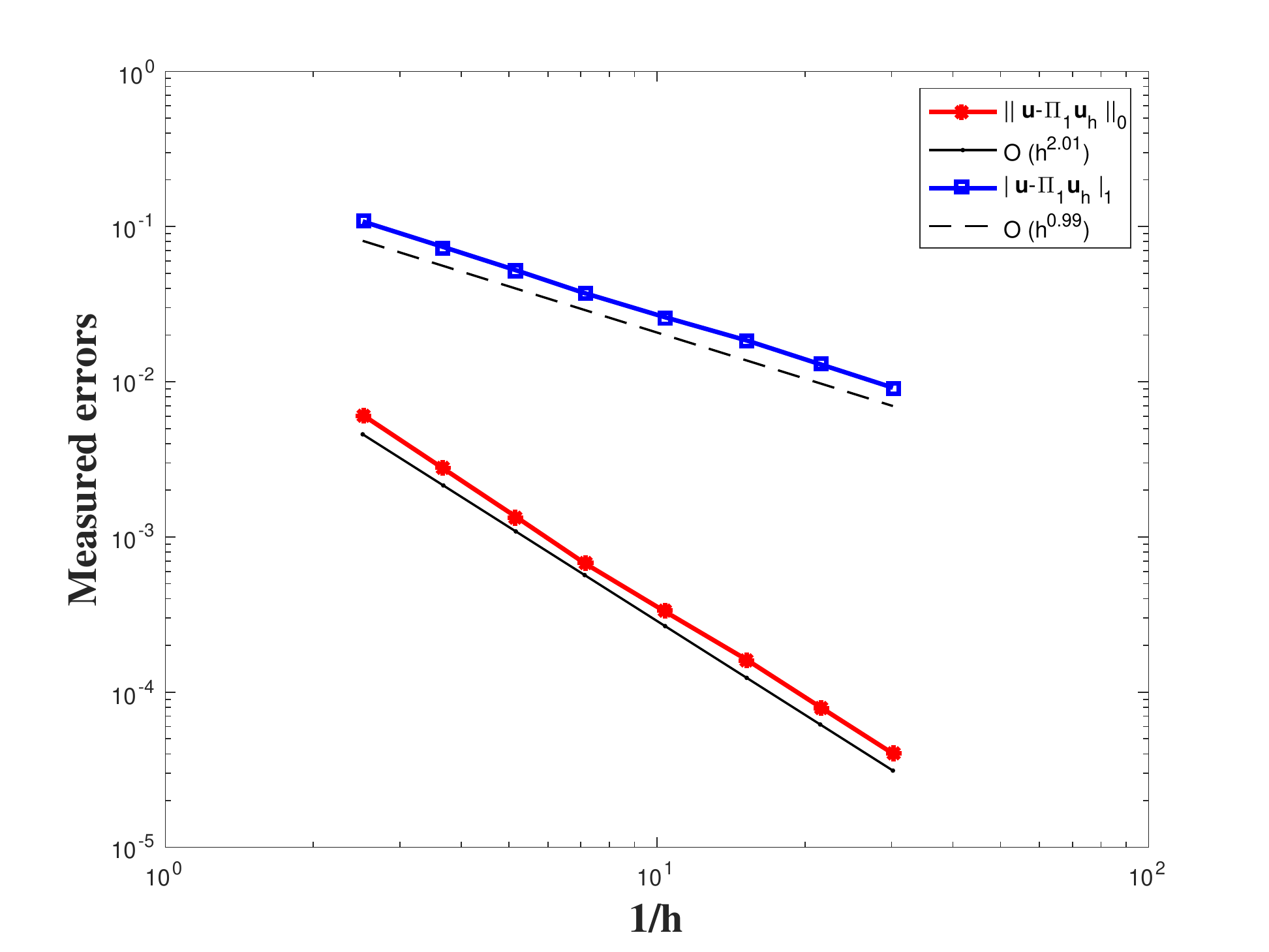}}
	\caption{$L^2$ and $H^1$ errors for $\lambda\in\{1,\,10^4,\,10^7,10^{10}\}$ on polygonal meshes $\mathcal{M}_3$ in Test 2.
	}
	\label{Fig9_Test2a_Lambda_ConvRates}
\end{figure}

\subsection{The completely incompressible limiting problem}
\label{Sec7-4}
{\bf Test 3:} Finally, we consider the completely incompressible problem with the exact solution given by
\begin{align*}
\bb{u}(x_1,\,x_2) =
\begin{pmatrix}
-{\sin}^3{(\pi{x_1})}{\sin}{(2\pi{x_2})} {\sin}{(\pi{x_2})} \\
{\sin}{(2\pi{x_1})}{\sin}{(\pi{x_1})} {\sin}^3{(\pi{x_2})}
\end{pmatrix}.
\end{align*}
It's easy to check that ${\rm div}~\boldsymbol{u}=0$ and $\boldsymbol{f}$ is independent of $\lambda$.
Now we investigate the impact of Lam\'{e} constant $\lambda$ on two different sequences of meshes
composed by Voronoi and distortion polygons, respectively.
Let $\mu=1$ and $\lambda=10^{10}$.

\begin{figure}[H]
	\centering
		\includegraphics[scale=0.5]{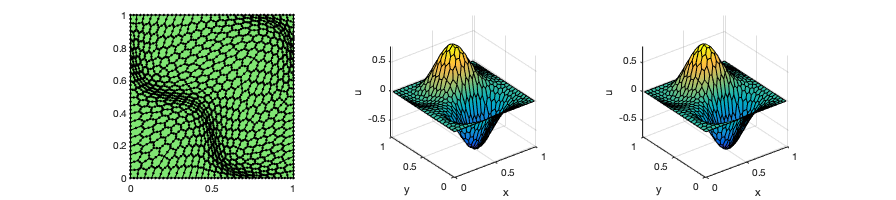} 
	\caption{Numerical and exact solutions for $\lambda=10^{10}$ on a distortion polygonal mesh in Test 3.
	(left: Mesh512, middle: numerical solution, right: exact solution)
	}
	\label{Fig11_Test3_SolutionInfo}
\end{figure}

The numerical solutions are well matched with the exact ones even on distortion polygonal meshes as observed
in Fig. \ref{Fig11_Test3_SolutionInfo}.
From the convergence graphs of the errors {\rm ErrH1} and {\rm ErrL2} for $\lambda=10^{10}$ on two families of polygonal meshes
displayed in Fig. \ref{Fig12_Test3_Lambda_ConvRates}, we still obtain a satisfactory result that the optimal rates of convergence are achieved
for both subdivisions in the completely incompressible limiting case.
\begin{figure}[H]
	\centering
		\subfigure[Voronoi  polygons]
	{\includegraphics[scale=0.37]{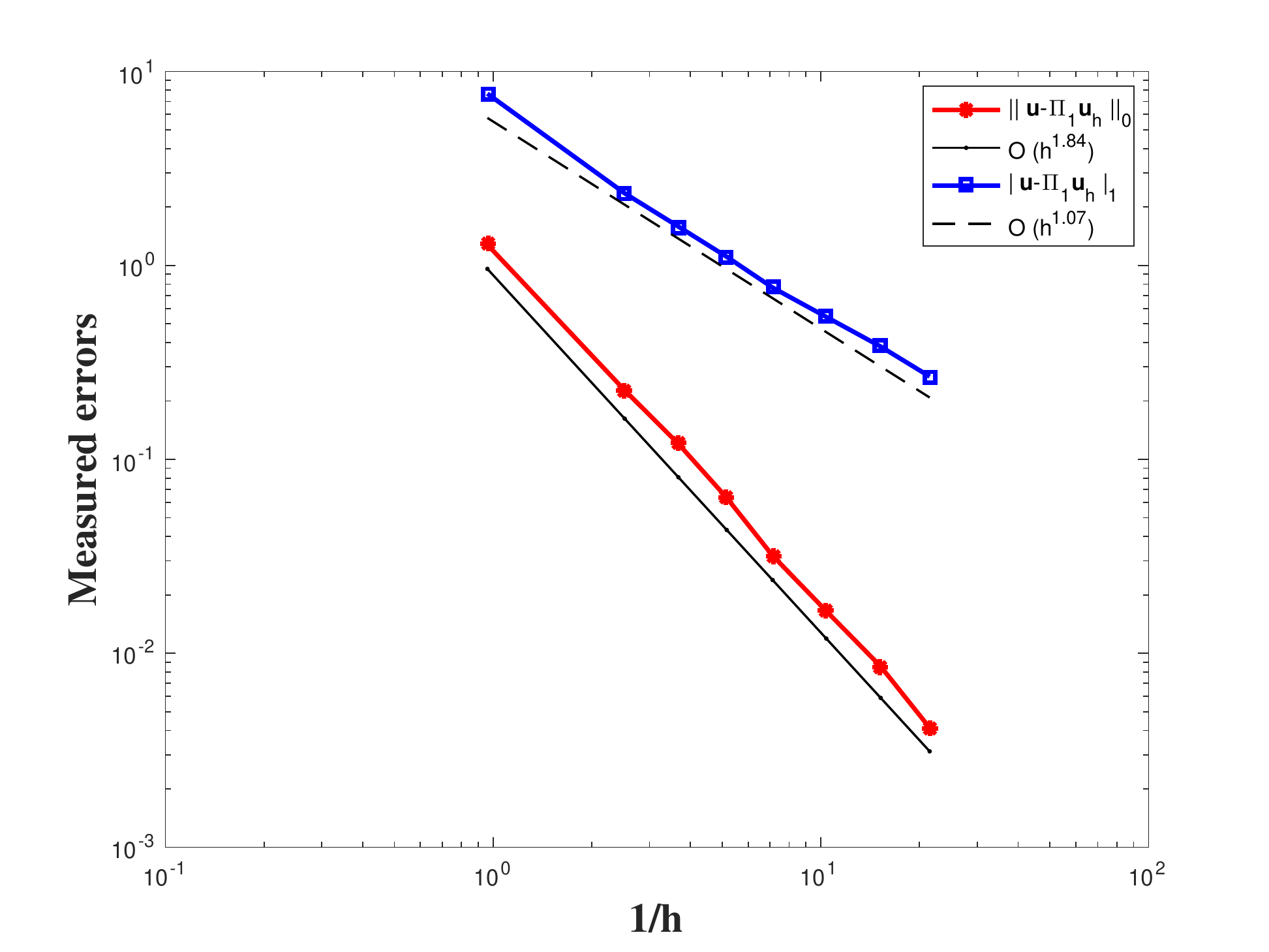}}
	\subfigure[Distortion polygons]
	{\includegraphics[scale=0.37]{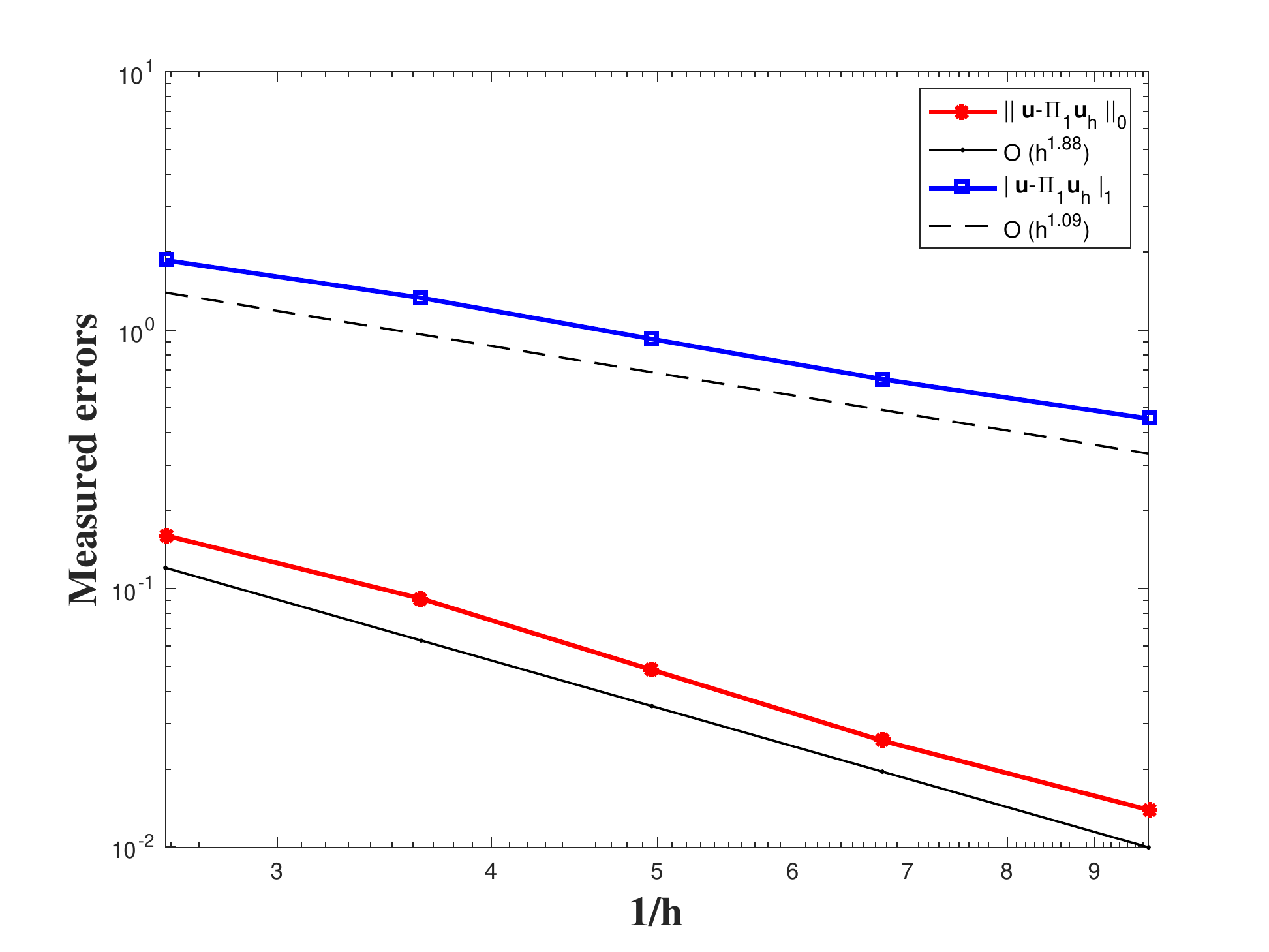}}
	\caption{$L^2$ and $H^1$ errors for $\lambda=10^{10}$ on Voronoi and distortion polygonal meshes in Test 3.
	}
	\label{Fig12_Test3_Lambda_ConvRates}
\end{figure}

\section{Conclusions}
\label{Sec8}
We propose a novel locking-free conforming VEM in the lowest order case for solving linear elasticity problems.
The key is to view the mesh $K$ as $\widetilde{K}$ which is formed by including an interior point of each edge as an additional vertex.
Then the method in \cite{Beirao-Brezzi-Marini-2013} can be extended to the case of $k=1$, if $V_1(K)$ is replaced by $V_1(\widetilde{K})$.
The parameter-free property and the optimal error estimate are established. Numerical results are consistent with theoretical findings.


\section*{References}

\begin{thebibliography}{10}

\bibitem{Adak-Nataraj-2020}
D.~Adak and S.~Natarajan.
\newblock Virtual element method for semilinear sine-{G}ordon equation over
  polygonal mesh using product approximation technique.
\newblock {\em Math. Comput. Simulation}, 172:224--243, 2020.

\bibitem{Adams1975}
R.~A. Adams.
\newblock {\em Sobolev Spaces}.
\newblock Academic Press, New York, 1975.

\bibitem{Ahmad-Alsaedi-Brezzi-Marini-2013}
B.~Ahmad, A.~Alsaedi, F.~Brezzi, L.~D. Marini, and A.~Russo.
\newblock Equivalent projectors for virtual element methods.
\newblock {\em Comput. Math. Appl.}, 66(3):376--391, 2013.

\bibitem{Artioli-Miranda-Lovadina-Patruno-2017}
E.~Artioli, S.~de~Miranda, C.~Lovadina, and L.~Patruno.
\newblock A stress/displacement virtual element method for plane elasticity
  problems.
\newblock {\em Comput. Methods Appl. Mech. Engrg.}, 325:155--174, 2017.

\bibitem{Artioli-Miranda-Lovadina-Patruno-2018}
E.~Artioli, S.~de~Miranda, C.~Lovadina, and L.~Patruno.
\newblock A family of virtual element methods for plane elasticity problems
  based on the {H}ellinger-{R}eissner principle.
\newblock {\em Comput. Methods Appl. Mech. Engrg.}, 340:978--999, 2018.

\bibitem{Artioli-Miranda-Lovadina-Patruno-2020}
E.~Artioli, S.~de~Miranda, C.~Lovadina, and L.~Patruno.
\newblock A dual hybrid virtual element method for plane elasticity problems.
\newblock {\em ESAIM Math. Model. Numer. Anal.}, 54(5):1725--1750, 2020.

\bibitem{Babuska-Suri-1992a}
I.~Babu\v{s}ka and M.~Suri.
\newblock Locking effects in the finite element approximation of elasticity
  problems.
\newblock {\em Numer. Math.}, 62(4):439--463, 1992.

\bibitem{Babuska-Suri-1992b}
I.~Babu\v{s}ka and M.~Suri.
\newblock On locking and robustness in the finite element method.
\newblock {\em SIAM J. Numer. Anal.}, 29(5):1261--1293, 1992.

\bibitem{Babuska-Szabo-1982}
I.~Babu\v{s}ka and B.~Szabo.
\newblock On the rates of convergence of the finite element method.
\newblock {\em Internat. J. Numer. Methods Engrg.}, 18(3):323--341, 1982.

\bibitem{Beirao-Brezzi-Cangiani-Manzini-2013}
L.~{Beir{\~{a}}o da Veiga}, F.~Brezzi, A.~Cangiani, G.~Manzini, L.~D. Marini,
  and A.~Russo.
\newblock Basic principles of virtual element methods.
\newblock {\em Math. Models Meth. Appl. Sci.}, 23(1):199--214, 2013.

\bibitem{Beirao-Brezzi-Marini-2013}
L.~{Beir{\~{a}}o da Veiga}, F.~Brezzi, and L.~D. Marini.
\newblock Virtual elements for linear elasticity problems.
\newblock {\em SIAM J. Numer. Anal.}, 51(2):794--812, 2013.

\bibitem{Beirao-Brezzi-Marini-Russo-2014}
L.~{Beir{\~{a}}o da Veiga}, F.~Brezzi, L.~D. Marini, and A.~Russo.
\newblock The hitchhiker's guide to the virtual element method.
\newblock {\em Math. Models Meth. Appl. Sci.}, 24(8):1541--1573, 2014.

\bibitem{Beirao-Lovadina-Russo-2017}
L.~{Beir{\~{a}}o da Veiga}, C.~Lovadina, and A.~Russo.
\newblock Stability analysis for the virtual element method.
\newblock {\em Math. Models Methods Appl. Sci.}, 27(13):2557--2594, 2017.

\bibitem{Berbatov-Lazarov-Jivkov-2021}
K.~Berbatov, B.~S. Lazarov, and A.~P. Jivkov.
\newblock A guide to the finite and virtual element methods for elasticity.
\newblock {\em Appl. Numer. Math.}, 169:351--395, 2021.

\bibitem{Bernardi-Raugel-1985}
C.~Bernardi and G.~Raugel.
\newblock Analysis of some finite elements for the {S}tokes problem.
\newblock {\em Math. Comp.}, 44(169):71--79, 1985.

\bibitem{BoffiBrezziFortin2013}
D.~Boffi, F.~Brezzi, and M.~Fortin.
\newblock {\em Mixed finite element methods and applications}, volume~44 of
  {\em Springer Series in Computational Mathematics}.
\newblock Springer, Heidelberg, 2013.

\bibitem{BrennerScott2008}
S.~C. Brenner and L.~R. Scott.
\newblock {\em The Mathematical Theory of Finite Element Methods}.
\newblock Springer, New York, 3rd edition, 2008.

\bibitem{Brenner-Sung-1992}
S.~C. Brenner and L.~Sung.
\newblock Linear finite element methods for planar linear elasticity.
\newblock {\em Math. Comp.}, 59(200):321--338, 1992.

\bibitem{Caceres-Gatica-Sequeira-2019}
E.~C\'{a}ceres, G.~N. Gatica, and F.~A. Sequeira.
\newblock A mixed virtual element method for a pseudostress-based formulation
  of linear elasticity.
\newblock {\em Appl. Numer. Math.}, 135:423--442, 2019.

\bibitem{Carstensen-Schedensack-2015}
C.~Carstensen and M.~Schedensack.
\newblock Medius analysis and comparison results for first-order finite element
  methods in linear elasticity.
\newblock {\em IMA J. Numer. Anal.}, 35(4):1591--1621, 2015.

\bibitem{Chen-Huang-2018}
L.~Chen and J.~Huang.
\newblock Some error analysis on virtual element methods.
\newblock {\em Calcolo}, 55(1):Paper No. 5, 23, 2018.

\bibitem{Chen-Huang-2020}
L.~Chen and X.~Huang.
\newblock Nonconforming virtual element method for {$2m$}-th order partial
  differential equations in {$R^n$}.
\newblock {\em Math. Comput.}, 89(324):1711--1744, 2020.

\bibitem{Costabel-Dauge-2015}
M.~Costabel and M.~Dauge.
\newblock On the inequalities of {Babu{\v{s}}ka-Aziz, Friedrichs and
  Horgan-Payne}.
\newblock {\em Arch. Ration. Mech. Anal.}, 217(3):873--898, 2015.

\bibitem{Dassi-Lovadina-Visinoni-2020}
F.~Dassi, C.~Lovadina, and M.~Visinoni.
\newblock A three-dimensional {H}ellinger-{R}eissner virtual element method for
  linear elasticity problems.
\newblock {\em Comput. Methods Appl. Mech. Engrg.}, 364:112910, 17, 2020.

\bibitem{Dassi-Vacca-2020}
F.~Dassi and G.~Vacca.
\newblock Bricks for the mixed high-order virtual element method: projectors
  and differential operators.
\newblock {\em Appl. Numer. Math.}, 155:140--159, 2020.

\bibitem{Dhanush-Natarajan-2019}
V.~Dhanush and S.~Natarajan.
\newblock Implementation of the virtual element method for coupled
  thermo-elasticity in {A}baqus.
\newblock {\em Numer. Algorithms}, 80(3):1037--1058, 2019.

\bibitem{Gain-Talischi-Paulino-2014}
A.~L. Gain, C.~Talischi, and G.~H. Paulino.
\newblock On the virtual element method for three-dimensional linear elasticity
  problems on arbitrary polyhedral meshes.
\newblock {\em Comput. Methods Appl. Mech. Engrg.}, 282:132--160, 2014.

\bibitem{Guo-Huang-2020}
Y.~Guo and J.~Huang.
\newblock A robust finite element method for elastic vibration problems.
\newblock {\em Comput. Methods Appl. Math.}, 20(3):481--500, 2020.

\bibitem{Harper-Wang-Liu-Tavener-Zhang-2020}
G.~Harper, R.~Wang, J.~Liu, S.~Tavener, and R.~Zhang.
\newblock A locking-free solver for linear elasticity on quadrilateral and
  hexahedral meshes based on enrichment of {L}agrangian elements.
\newblock {\em Comput. Math. Appl.}, 80(6):1578--1595, 2020.

\bibitem{Huang-Lin-2021}
J.~Huang and S.~Lin.
\newblock A posteriori error analysis of a non-consistent virtual element
  method for reaction diffusion equations.
\newblock {\em Appl. Math. Lett.}, 122:Paper No. 107531, 10, 2021.

\bibitem{Huang-2020}
X.~Huang.
\newblock Nonconforming virtual element method for 2{$m$}th order partial
  differential equations in {$\mathbb{R}^n$} with {$m>n$}.
\newblock {\em Calcolo}, 57(4):42, 2020.

\bibitem{Huang-Huang-2013}
X.~Huang and J.~Huang.
\newblock The compact discontinuous {G}alerkin method for nearly incompressible
  linear elasticity.
\newblock {\em J. Sci. Comput.}, 56(2):291--318, 2013.

\bibitem{Kwak-Park-2022}
D.~Y. Kwak and H.~Park.
\newblock Lowest-order virtual element methods for linear elasticity problems.
\newblock {\em Comput. Methods Appl. Mech. Engrg.}, 390:Paper No. 114448, 2022.

\bibitem{Mora-Rivera-2020}
D.~Mora and G.~Rivera.
\newblock A priori and a posteriori error estimates for a virtual element
  spectral analysis for the elasticity equations.
\newblock {\em IMA J. Numer. Anal.}, 40(1):322--357, 2020.

\bibitem{Reddy-Huyssteen-2019}
B.~D. Reddy and D.~van Huyssteen.
\newblock A virtual element method for transversely isotropic elasticity.
\newblock {\em Comput. Mech.}, 64(4):971--988, 2019.

\bibitem{Talischi-Paulino-Pereira-2012}
C.~Talischi, G.~H. Paulino, A.~Pereira, and I.~F.~M. Menezes.
\newblock Polymesher: a general-purpose mesh generator for polygonal elements
  written in {M}atlab.
\newblock {\em Struct. Multidiscip. Optim.}, 45(3):309--328, 2012.

\bibitem{Tang-Liu-Zhang-Feng-2020}
X.~Tang, Z.~Liu, B.~Zhang, and M.~Feng.
\newblock A low-order locking-free virtual element for linear elasticity
  problems.
\newblock {\em Comput. Math. Appl.}, 80(5):1260--1274, 2020.

\bibitem{Wang-Wang-Wang-Zhang-2016}
C.~Wang, J.~Wang, R.~Wang, and R.~Zhang.
\newblock A locking-free weak {G}alerkin finite element method for elasticity
  problems in the primal formulation.
\newblock {\em J. Comput. Appl. Math.}, 307:346--366, 2016.

\bibitem{Zhang-Feng-2018}
B.~Zhang and M.~Feng.
\newblock Virtual element method for two-dimensional linear elasticity problem
  in mixed weakly symmetric formulation.
\newblock {\em Appl. Math. Comput.}, 328:1--25, 2018.

\bibitem{Zhang-Zhao-Yang-Chen-2019}
B.~Zhang, J.~Zhao, Y.~Yang, and S.~Chen.
\newblock The nonconforming virtual element method for elasticity problems.
\newblock {\em J. Comput. Phys.}, 378:394--410, 2019.

\bibitem{Zhang-1997}
Z.~Zhang.
\newblock Analysis of some quadrilateral nonconforming elements for
  incompressible elasticity.
\newblock {\em SIAM J. Numer. Anal.}, 34(2):640--663, 1997.

\end{thebibliography}

\end{document}